\documentclass[12pt]{amsart}

\usepackage[margin=1.3in]{geometry}

\usepackage{amsmath,amssymb, enumerate}
\usepackage[dvipdfmx]{graphicx}
\usepackage{graphicx, tabularx}

\begin{document}


\newfont{\ef}{eufm10 scaled\magstep1}
\newfont{\efs}{eufm8}
\newfont{\bbl}{msbm10 scaled\magstep1}
\newfont{\cl}{cmsy10 scaled\magstep1}
\newfont{\cls}{cmsy8}

%
\newcommand{\St}{\mbox{\ef T}}
\newcommand{\XG}{\mbox{\ef X}}
\newcommand{\CG}{\mbox{\ef C}}
\newcommand{\DG}{\mbox{\ef D}}
\newcommand{\UG}{\mbox{\ef U}}
\newcommand{\UGs}{\mbox{\efs U}}
\newcommand{\LG}{\mbox{\ef L}}
\newcommand{\MG}{\mbox{\ef M}}
\newcommand{\NGG}{\mbox{\ef N}}
\newcommand{\PG}{\mbox{\ef P}}
\newcommand{\WG}{\mbox{\ef W}}
\newcommand{\XGG}{\mbox{\ef X}}
\newcommand{\XGGs}{\mbox{\efs X}}
\newcommand{\Zh}{\mbox{\bbl Z}}
\newcommand{\Nat}{\mbox{\bbl N}}
\newcommand{\Rd}{\mbox{\cl R}}
\newcommand{\Gh}{\mbox{\cl G}}
\newcommand{\Ghs}{\mbox{\cls G}}
\newcommand{\HGh}{\mbox{\cl H}}
\newcommand{\HGhs}{\mbox{\cls H}}
\newcommand{\Sy}{\mbox{\cl S}}
\newcommand{\Sys}{\mbox{\cls S}}
\newcommand{\Ty}{\mbox{\cl T}}
\newcommand{\Tys}{\mbox{\cls T}}
\newcommand{\Nc}{\mbox{\cl N}}
\newcommand{\Wsy}{\mbox{\cl W}}
\newcommand{\Vt}{\mbox{\cl V}}
\newcommand{\Vts}{\mbox{\cls V}}
\newcommand{\Ed}{\mbox{\cl E}}
\newcommand{\Bs}{\mbox{\cl B}}
\newcommand{\Ms}{\mbox{\cl M}}
\newcommand{\Ds}{\mbox{\cl D}}

\newtheorem{theorem}{Theorem}[section]
\newtheorem{lemma}[theorem]{Lemma}
\newtheorem{proposition}[theorem]{Proposition}
\newtheorem{corollary}[theorem]{{\bf Corollary}}

\theoremstyle{definition}
\newtheorem{definition}[theorem]{Definition}
\newtheorem{example}[theorem]{Example}
\newtheorem{xca}[theorem]{Exercise}
\newtheorem{remark}[theorem]{Remark}

\theoremstyle{remark}
\newtheorem{claim}[theorem]{Claim}


\newcommand{\R}{\mathbb{R}}
\newcommand{\C}{\mathbb{C}}
\newcommand{\N}{\mathbb{N}}
\newcommand{\Q}{\mathbb{Q}}
\newcommand{\Z}{\mathbb{Z}}
\newcommand{\F}{\mathbb{F}}


\newcommand{\tvec}[2]{\left(\begin{array}{c}{#1}\\{#2}\end{array}\right)}
\newcommand{\hvec}[3]{\left(\begin{array}{c}{#1}\\{#2}\\{#3}\end{array}\right)}
\newcommand{\fvec}[4]{\left(\begin{array}{c}{#1}\\{#2}\\{#3}\\{#4}\end{array}\right)}
\newcommand{\tma}[4]{\left(\begin{array}{cc} {#1} & {#2} \\ {#3} & {#4} \end{array}\right)}
\newcommand{\hma}[9]{\left(\begin{array}{ccc} {#1} & {#2} & {#3} \\ {#4} & {#5} & {#6} \\ {#7} & {#8} & {#9} \end{array}\right)}


\newcommand{\bi}[2]{\binom{{#1}}{{#2}}}
\newcommand{\stir}[2]{\left\{\begin{array}{c}{#1}\\{#2}\end{array}\right\}}

\newcommand{\fun}[5]{\begin{array}{cccc}{#1}:&{#2}&\to&{#3}\\&{#4}&\mapsto&{#5}\end{array}}

\newcommand{\abs}[1]{\left|{#1}\right|} 
\newcommand{\brac}[1]{\left[{#1}\right]} 
\newcommand{\set}[1]{\left\{{#1}\right\}} 
\newcommand{\setof}[2]{\left\{{#1}\,:\,{#2}\right\}}
\newcommand{\pr}[1]{\left({#1}\right)} 
\newcommand{\piece}[4]{\left\{\begin{array}{cc} {#1} & \textrm{if }{#2} \\ {#3} & \textrm{if }{#4} \end{array}\right.}
\newcommand{\otherpiece}[3]{\left\{\begin{array}{cc} {#1} & \textrm{if }{#2} \\ {#3} & \textrm{otherwise} \end{array}\right.}
\newcommand{\norm}[1]{\left\Vert{#1}\right\Vert}
\newcommand{\ceil}[1]{\left\lceil{#1}\right\rceil}
\newcommand{\floor}[1]{\left\lfloor{#1}\right\rfloor}
\newcommand{\ip}[2]{\left\langle{#1},{#2}\right\rangle}

\renewcommand{\b}[1]{\mathbf{{#1}}}
\newcommand{\m}[1]{\mathcal{{#1}}}
\newcommand{\mb}[1]{\mathhbb{{#1}}}

\renewcommand{\t}[1]{\textnormal{{#1}}}
\newcommand{\tr}[1]{\textrml{{#1}}}
\newcommand{\num}[1]{\#(\textnormal{{#1}})}


\newcommand{\dis}{\displaystyle}
\newcommand{\hor}{\line(1,0){500}}
\newcommand{\np}{\newpage}
\newcommand{\q}{\quad}
\newcommand{\qq}{\qquad}


\newcommand{\sub}{\subseteq}
\newcommand{\nsub}{\subsetneq}

\newcommand{\del}{\bigtriangledown}
\renewcommand{\P}{\mathbb{P}}
\newcommand{\E}{\mathbb{E}}
\newcommand{\V}{\mathbb{V}}
\newcommand{\I}{\mathbb{I}}
\newcommand{\Cov}{\textrm{Cov}}
\newcommand{\Var}{\textnormal{Var}}

\newcommand{\congp}{\equiv_p}
\newcommand{\congf}{\equiv_4}
\newcommand{\iso}{\cong}

\newcommand{\ops}{\overset{ops}{\longleftrightarrow}}
\newcommand{\egf}{\overset{egf}{\longleftrightarrow}}
\newcommand{\dir}{\overset{Dir}{\longleftrightarrow}}

\newcommand{\sm}{\setminus}


\renewcommand{\d}{\delta}
\newcommand{\e}{\epsilon}
\renewcommand{\O}{\Omega}
\renewcommand{\o}{\omega}
\renewcommand{\l}{\lambda}
\newcommand{\g}{\gamma}
\newcommand{\D}{\Delta}


\newcommand{\cd}{\cdots}
\newcommand{\ld}{\ldots}
\newcommand{\vd}{\vdots}
\newcommand{\dd}{\ddots}
\newcommand{\cdpt}{\cdot}


\newcommand{\K}{K_{\d,n-\d}}
\newcommand{\s}{\sum}
\newcommand{\n}{[n]}






\title{Non-noetherian groups and primitivity of their group algebras}


\author[J.~Alexander]{James Alexander}
\author[T.~Nishinaka]{Tsunekazu Nishinaka}

\address[J.~Alexander]{
\begin{flushleft}
\hspace{0.3cm} Department of Mathematics\\
\hspace{0.3cm} University of Delaware\\
\hspace{0.3cm} Newark, DE 19716, United States \\
\end{flushleft}
}
\email{jamesja@udel.edu}

\address[T.~Nishinaka]{
\begin{flushleft}
\hspace{0.3cm} University of Hyogo\\
\hspace{0.3cm} 8-2-1 Gakuen Nishimachi Nishiku \\
\hspace{0.3cm} Kobe-City 651-2197, Japan
\end{flushleft}
}
\email{nishinaka@econ.u-hyogo.ac.jp}

\thanks{The second author partially supported by Grants-in-Aid for Scientific Research under grant no. 26400055}



\maketitle

\begin{abstract}
We prove that the group algebra 
$KG$ of a group $G$ over a field $K$ is primitive,
provided that $G$ has a non-abelian free subgroup with the same cardinality as $G$,
and that $G$ satisfies the following condition 
$(\ast)$:
for each subset $M$ of $G$ consisting of a
finite number of elements not equal to $1$, and for any positive integer $m$,
there exist distinct $a$, $b$, and $c$ in $G$
so that
if 
$(x_{1}^{-1}g_1x_{1})
\cdots (x_{m}^{-1}g_mx_{m})=1$, where $g_i$ is in $M$ and $x_i$ is equal to $a$, $b$, or $c$ for all $i$ between $1$ and $m$, then $x_{i}=x_{i+1}$ for some $i$. This generalizes results of \cite{Bal}, \cite{For}, \cite{Ni07}, and \cite{Ni11}, and proves that, for every countably infinite group $G$ satisfying $(\ast)$, $KG$ is primitive for any field $K$. We use this result to determine the primitivity of group algebras of one relator groups with torsion. 
\end{abstract}




\section{Introduction}
\label{int}

A ring $R$ is said to be (right) primitive if it contains a faithful irreducible (right) $R$-module,
or equivalently, if  there exists a maximal (right) ideal in $R$ which includes no non-trivial ideal of $R$. The main purpose of this work is to determine, as generally as possible, for which fields $K$ the  group algebra $KG$ of a non-noetherian group $G$ is primitive. 

The study of the primitivity of general group algebras has been a topic of much interest over the last few decades. In 1978, by a series of studies by 
Domanov \cite{Dom78}, Farkas-Passman \cite{Far-Ps78},
and Roseblade \cite{Rosb78},
a complete classification of the primitivity of group algebras of
polycyclic-by-finite groups was given.
In particular, it was determined that, 
for a polycyclic-by-finite group $G$,
the group algebra $KG$ is primitive
if and only if its FC-center is trivial
and $K$ is not an absolute field.
These groups
belong to the class of noetherian groups.
We note that it is known to be difficult to find a noetherian group
which is not polycyclic-by-finite (see \cite{Ols}), and that almost all other known infinite groups 
belong to the class of non-noetherian groups, including free groups, locally free groups, free products,
amalgamated free products, HNN-extensions, Fuchsian groups, one relator groups, and free Burnside groups. 

In 1973,
Formanek \cite{For} showed that $KG$ is primitive for any field $K$,
provided that $G$ is the free product of non-trivial groups $A$ and $B$,
both of which are not isomorphic to the infinite dihedral group.
In 1989,  Balogun \cite{Bal} generalized this result
to one for amalgamated free products.
Since then, author Nishinaka
gave two results on primitivity of group rings $KG$, one 
in 2007 \cite{Ni07}, and another in 2011 \cite{Ni11};
one is a result for the ascending HNN-extension $G$ of a non-abelian free group,
and the other is for a locally free group $G$.
In this work, we will give a result which generalizes these main results of \cite{Bal}, \cite{For}, \cite{Ni07}, and \cite{Ni11}.
Consider the following condition:
\vskip3pt

\begin{tabularx}{11.5cm}{lX}
$(\ast)$
&For each subset $M$ of $G$ consisting of a
finite number of elements not equal to $1$, and for any positive integer $m$,
there exist distinct $a$, $b$, and $c$ in $G$
so that
if 
$(x_{1}^{-1}g_1x_{1})
\cdots (x_{m}^{-1}g_mx_{m})=1$, where $g_i$ is in $M$ and $x_i$ is equal to $a$, $b$, or $c$ for all $i$ between $1$ and $m$, then $x_{i}=x_{i+1}$ for some $i$.\\
\end{tabularx}
\vskip3pt

We will explain that if $G$ is a countably infinite group which satisfies $(\ast)$,
then $KG$ is primitive for any field $K$.
More generally,
we will prove the following theorem:
\begin{theorem}\label{MR_TH}
Let $G$ be a group
which has a non-abelian free subgroup whose cardinality is the same as that of $G$, and suppose that $G$ satisfies $(\ast)$. Then, if $R$ is a domain with $|R|\leq |G|$, the group ring $RG$ of $G$ over $R$ is primitive. In particular, the group algebra $KG$ is primitive for any field $K$.
\end{theorem}

As we discuss in Section~\ref{PAFP}, one can easily check that non-noetherian groups with free subgroups often satisfy $(\ast)$; e.g., non-abelian locally free groups, amalgamated free products, or HNN-extensions will satisfy $(\ast)$. Moreover,
using Theorem~\ref{MR_TH}, we will show that every group algebra of a one relator group with torsion is primitive.

In order to prove Theorem~\ref{MR_TH},
we construct a maximal right ideal in $KG$
which includes no non-trivial ideal of $KG$.
We then show that the constructed right ideal
is proper.
To do this, we use graph-theoretic methods.
In particular,
we define what we call an SR-graph and an SR-cycle in Section~\ref{SRgraph}, and show in Section~\ref{MF} that 
the proof of Theorem~\ref{MR_TH} can be reduced to
finding an SR-cycle in a certain SR-graph. Applications of Theorem~\ref{MR_TH} and future work are then discussed in subsequent sections. 

For the remainder of this document, let $\N$ denote the set of positive integers, $\N_0:=\N\cup\{0\}$, and $[n]:=\{1,2,\ld,n\}$ for any $n\in\N$. As mentioned, the results of Section~\ref{SRgraph} are established graph-theoretically. We do not assume much prior knowledge of graph theory, however, some basic familiarity is assumed; for any terminology and notation which we do not define as it is assumed particularly standard, we follow \cite{B-M} (which can also serve as an introductory text if needed). Though it is nonstandard in modern graph theory, we use script letters to denote graphs so that our notation for graphs is easily distinguishable from our notation for groups.


\section{SR-graphs}
\label{SRgraph}

In this section, we define an SR-graph and an SR-cycle; we show that certain SR-graphs have SR-cycles. We write $\m{G}=(V,E)$ to denote that $\m{G}$ is a simple graph (undirected and without loops or multi-edges) having vertex set $V$ and edge set $E$. We denote $\{v,w\}\in E$ by $vw$ when there is no risk of confusion. We let $I(\m{G})$ denote the isolated vertices of $\m{G}$, i.e., the set of all $v\in V$ for which $vw\notin E$ for all $w\in V$. We denote by $C(\m{G})$ the set of components of $\m{G}$, i.e., the set of subgraphs of $\m{G}$ which partition $\m{G}$, so that in each subgraph any two vertices are joined by a path, and so that no vertices which do not lie in the same subgraph are joined by a path in $\m{G}$; we let $c(\m{G}):=|C(\m{G})|$. We say that $\m{G}$ is connected if $c(\m{G})=1$. For any $W\sub V$, we let $\m{G}[W]$ denote the subgraph of $\m{G}$ induced by $W$, i.e., $\m{G}[W]:=(W,\{vw\in E\,|\,v,w\in W\})$; let $\m{G}_v:=\m{G}[V\sm\{v\}]$. We let $X(\m{G})$ denote the set of all cut-vertices of $\m{G}$, i.e., the set of all $v\in V$ so that $c({\m{G}_v})>c(\m{G})$. We begin with two definitions:

\begin{definition}\label{def:SRdefnew}
Let $\m{G}:=(V,E)$ and $\m{H}:=(V,F)$. If every component of $\m{G}$ is a complete graph, and if $E\cap F=\emptyset$, then we call the triple $\m{S}=(V,E,F)$ a \emph{sprint relay graph}, abbreviated SR-graph. We view $\m{S}$ as the graph $(V,E\cup F)$, guaranteed simple as $E\cap F=\emptyset$, with edges partitioned into $E$ and $F$; we denote $\m{S}$ by $(\m{G},\m{H})$ rather than $(V,E,F)$ when convenient. 
\end{definition}

\begin{definition}\label{def:SRcycledefnew}
A cycle in an SR-graph $(V,E,F)$ is called an SR-cycle if its edges belong alternatively to $E$ and not to $E$; more formally, we call cycle $(V',E')$ an SR-cycle if there is labeling $V'=\{v_1,v_2,\ld,v_c\}$ and $E'=\{v_1v_2,v_2v_3,\ld,v_{c-1}v_c,v_c v_1\}$ so that $v_iv_{i+1}\in E$ if and only if $i$ is odd, for some even $c$. 
\end{definition}

The class of SR-graphs is a subclass of the class
of two-edge coloured graphs
in which an SR-cycle is called an alternating cycle (see \cite{GROSSMAN}).

For the remainder of this section, fix $\m{S}=(V,E,F)$, $\m{G}=(V,E)$, and $\m{H}=(V,F)$ so that $V\neq\emptyset$, every component of $\m{G}$ complete, and $\m{S}$ an SR-graph. 
Moreover, let $\m{H}_1,\m{H}_2,\ld,\m{H}_n$ denote the components of $\m{H}$ with $\m{H}_i=(V_i,E_i)$ over $i\in[n]$.  
We first address the case in which $\m{H}_i$ is a complete graph for each $i\in[n]$ as follows:

\begin{theorem}\label{thm:complete}
If $\m{S}$ is connected and each component of $\m{H}$ is complete, then $\m{S}$ has an SR-cycle if and only if $c(\m{G})+c(\m{H})<|V|+1$. 
\end{theorem}

Consider the following result of Grossman and H{\"a}ggkvist \cite{GROSSMAN}:

\begin{lemma}\label{lem:gross}
If $\m{S}$ has no SR-cycle, then $I(\m{G})\cup I(\m{H})\cup X(\m{S})\neq\emptyset$. 
\end{lemma}

Before moving on, let us  collect some straightforward observations:

\begin{remark}\label{rem:pre}
Assume that $\m{S}$, $\m{G}$, and $\m{H}$ satisfy the hypotheses of Theorem~\ref{thm:complete}. 
\begin{enumerate}[(I)]
\item If $v\notin X(\m{S})$, then \label{1}
\begin{enumerate}[(i)]
\item $v\in I(\m{G})\cup I(\m{H})$ implies $c({\m{G}_v})+c({\m{H}_v})=c(\m{G})+c(\m{H})-1$;\label{1a}
\item $v\notin I(\m{G})\cup I(\m{H})$ implies $c({\m{G}_v})=c(\m{G})$ and $c({\m{H}_v})=c(\m{H})$.\label{1b}
\end{enumerate}
\item If $v\in X(\m{S})$, then without loss of generality, \label{2}
\begin{enumerate}[(i)]
\item ${\m{S}_v}$ is an SR-graph with components $(\m{G}_1,\m{H}_1)$ and $(\m{G}_2,\m{H}_2)$;\label{2a}
\item $\sum_{i=1}^2(c(\m{G}_i)+c(\m{H}_i))=c(\m{G})+c(\m{H})$ and $|V_1|+|V_2|=|V|-1$, where $V_1$ and $V_2$ are the vertex sets of $(\m{G}_1,\m{H}_1)$ and $(\m{G}_2,\m{H}_2)$, respectively.\label{2b}
\end{enumerate}
\end{enumerate}
\end{remark}

We are now ready to prove Theorem~\ref{thm:complete}. 

\begin{proof}[Proof of Theorem~\ref{thm:complete}]
Before entering the heart of this proof, we show that 
\begin{equation}
c(\m{G})+c(\m{H})\leq |V|+1, \label{eq:lem}
\end{equation}
which holds trivially when $|V|=1$. Assume, by way of induction, that $|V|>1$ and that (\ref{eq:lem}) holds for SR-graphs on fewer vertices. Fix $v\in V$. If $v\notin X(\m{S})$, then ${\m{S}_v}$ is connected and ${\m{H}_v}$ has complete components; thus, $c(\m{G}_v)+c(\m{H}_v)\leq |V|$ by induction, and so (\ref{eq:lem}) follows from Remark~\ref{rem:pre}(\ref{1}). If $v\in X(\m{S})$, then ${\m{S}_v}$ has components $(\m{G}_1,\m{H}_1)$ and $(\m{G}_2,\m{H}_2)$ by Remark~\ref{rem:pre}(\ref{2})(i); by induction, $c(\m{G}_i)+c(\m{H}_i)\leq |V_i|+1$ for $i\in[2]$, and thus (\ref{eq:lem}) holds by Remark~\ref{rem:pre}(\ref{2})(ii). 

We are now ready for the crux of our argument. First, assume that $\m{S}$ has an SR-cycle. We prove by induction on $|V|$ that $c(\m{G})+c(\m{H})<|V|+1$, noting that we may assume $|V|\geq 4$. This holds trivially if $|V|=4$, so assume $|V|>4$ and, by way of induction, that the the result holds for SR-graphs on fewer vertices. This result holds trivially if $\m{S}$ is an SR-cycle, so we may assume that there is $C\subsetneq V$ so that $\m{S}[C]$ is an SR-cycle. 

Consider $v\in V\sm C$.  If $v\notin X(\m{S})$, then we can obtain the desired result with a similar argument to that which we used in the first paragraph when $v\notin X(\m{S})$ was assumed. Assume $v\in X(\m{S})$, in which case $\m{S}_v$ has components $(\m{G}_1,\m{H}_1)$ and $(\m{G}_2,\m{H}_2)$ by Remark~\ref{rem:pre}(\ref{2})(i). Since $v\in X(\m{S})$ and $\m{G}$ and $\m{H}$ have complete components, either $C\sub V_1$ or $C\sub V_2$; say, without loss of generality, that $C\sub V_1$. Then, by our induction hypothesis, $c(\m{G}_1)+c(\m{H}_1)<|V_1|+1$. Also, by (\ref{eq:lem}), $c(\m{G}_2)+c(\m{H}_2)\leq |V_2|+1$. Thus, by Remark~\ref{rem:pre}(\ref{2})(ii) that $c(\m{G})+c(\m{H})<|V|+1$. 

To prove the converse, by (\ref{eq:lem}), it suffices to show that if $\m{S}$ has no SR-cycle, then $c(\m{G})+c(\m{H})=|V|+1$. To that end, assume $\m{S}$ has no SR-cycle. Our proof will again be by induction on $|V|$. If $X(\m{S})\neq\emptyset$ then we may consider $v\in X(\m{S})$ and obtain the result with a similar argument to that which we used in the first paragraph when $v\in X(\m{S})$ was assumed. Assume $X(\m{S})=\emptyset$. By Lemma~\ref{lem:gross}, there is $v\in I(\m{G})\cup I(\m{H})$. By induction, $c(\m{G}_v)+c(\m{H}_v)=|V|$. It follows from Remark~\ref{rem:pre}(\ref{1})(i) that $c(\m{G})+c(\m{H})=|V|+1$.  
\end{proof}

For the remainder of this section, let $I:=I(\m{G})$, $W:=V\sm I$, $W_i:=V_i\sm I$, and say $\m{H}[W_i]=(W_i,F_i)$. For any $m_1,m_2,\ld,m_k\in\N$, we let $K_{m_1,m_2,\ld,m_k}$ denote the complete multipartite graph with partite sets of size $m_1,m_2,\ld,m_k$, i.e., the graph $(V',E')$ so that $V'$ can be partitioned into sets $P_1,P_2,\ld,P_k$ called partite sets, with $|P_i|=m_i$ and $vw\in E'$ if and only if $v$ and $w$ are in different partite sets for all $v,w\in V$. We let $\mu(K_{m_1,m_2,\ld,m_k}):=\max_{i\in[k]}\{m_i\}$. We now handle the case in which each component of $\m{H}$ is complete multipartite.

\begin{theorem}\label{thm:multipart}
Assume that $\m{H}_i$ is a complete multipartite graph for each $i\in[n]$. If $|I|\leq n$ and $|V_i|>2\mu(\m{H}_i)$ for each $i\in[n]$, then $\m{S}$ has an SR-cycle. 
\end{theorem}

In order to build to a proof of Theorem~\ref{thm:multipart}, we first prove two lemmas. 

\begin{lemma}\label{Add_isol}
Let ${U}\sub V$ with ${U}\cap I=\emptyset$, and let ${U}':=V\sm {U}$. Then, $|I\cap {U}'|\leq |I(\m{G}[{U}'])|\leq |I\cap {U}'|+|{U}|$. 
\end{lemma}

\begin{proof}
As $I\cap {U}'\sub I(\m{G}[{U}'])$, the leftmost inequality is trivial. If $v\in I(\m{G}[{U}'])\sm (I\cap {U}')$, then there is $w\in {U}$ with $vw\in E$, by definition. Moreover, there cannot be more than one such $v$ for each such $w$; indeed, if we had $v,v'\in {U}'$ with $vw,v'w\in E$, then we would have $vv'\in E$ by the fact that $\m{G}$ has complete components, implying that $v\notin I(\m{G}[{U}'])$. So, $|I(\m{G}[{U}'])| - |I\cap {U}'| \leq |{U}|$. 
\end{proof}

\begin{lemma}\label{G_lem_kpart}
If $\m{H}[W_i]\not\simeq K_{1,m}$ for all $m\geq2$ and $I(\m{H}[W])=\emptyset$, then $\m{S}$ has an SR-cycle. 
\end{lemma}

\begin{proof}
We show, more strongly, that $\m{S}[W]$ has an SR-cycle. For ease of notation, assume that $\m{S}=\m{S}[W]$, i.e., that $I=\emptyset$, $I(\m{H})=\emptyset$, and $\m{H}_i\not\simeq K_{1,m}$ for $m\geq2$. Fix $v_1\in V$. As $I=\emptyset$, there is $w_1\in {V}\sm\{v_1\}$ with $v_1w_1\in E$. As $I(\m{H})=\emptyset$ and $E\cap F=\emptyset$, there is $v_2\in {V}\sm\{v_1,w_1\}$ with $w_1v_2\in F$. Since $I=\emptyset$, and since $\m{G}$ has complete components and $E\cap F=\emptyset$, there is $w_2\in {V}\sm\{v_1,w_2,v_2\}$ with $v_2w_2\in E$. Continuing this way until we no longer can (noting that $|{V}|<\infty$), we create an $E$- alternating path either of the form $\{v_1,w_1,\ld,w_i,v_{i+1}\}$ or of the form $\{v_1,w_1,\ld,v_i,w_i\}$. First assume the form $\{v_1,w_1,\ld,w_i,v_{i+1}\}$. By maximality and since $I=\emptyset$, either $v_{i+1}v_j\in E$ for some $j\in[i]$, or $v_{i+1}w_j\in E$. In the latter case, edges $w_jv_{j+1},\ld,w_iv_{i+1},v_{i+1}w_j$ form an SR-cycle. In the former case, since $v_jw_j\in E$ and $\m{G}$ has complete components, we must have $v_{i+1}w_j\in E$, thus reducing us to the latter case. 

It remains to assume that our $E$-alternating path has the form $\{v_1,w_1,\ld,v_i,w_i\}$. By similar reasoning to that used in the previous paragraph, either $w_iv_j\in F$
or $w_iw_j\in F$ for some $j\in[i]$. In the former case, edges $v_{j}w_j,\ld,v_iw_i,w_{i}v_j$ form an SR-cycle. Assume $w_iw_j\in F$. We may assume that
$j$ is the minimum number
such that $w_iw_j\in F$ and, moreover, that
$w_iv_{j'}\not\in F$ for any $j'$.
Since $w_jv_{j+1}\in F$, there is a component $\m{H}_q=(V_q, E_q)$ of $\m{H}$
such that $w_i, w_j, v_{j+1}\in V_q$.  Since the subgraph of $\m{H}_q[\{w_i, w_j, v_{j+1}\}]\simeq K_{1,2}$ and $\m{H}_q\not\simeq K_{1,m}$ for $m\geq2$, there is $w\in {V}_q\sm\{w_i, w_j, v_{j+1}\}$. Since $w_iv_{j+1}\not\in F$ by assumption,
we  see that
there exists $h$ with $h\ne j$ such that $w_iw_h\in F$,
where $h>j$ by the minimality of $j$.
It then follows, as $\m{H}_q$ is complete-multipatite, 
that $w_hv_{j+1}\in F_q$; so, the edges $v_{j+1}w_{j+1},w_{j+1}v_{j+2},\ldots, w_{h}v_{j+1}$ form an SR-cycle. 
\end{proof}

We are now read to prove Theorem~\ref{thm:multipart}. 

\begin{proof}[Proof of Theorem~\ref{thm:multipart}]
Our proof is by induction on $n$. Assume $n=1$, and say $\m{H}_1$ has partite sets $P_1,P_2,\ld,P_p$. We note that if there are distinct $i,j\in[p]$, and $v_i,w_i\in P_i$ and $v_j,w_j\in P_j$ with $v_iw_i, v_jw_j\in E$, then $\m{S}[\{v_i,w_i,v_j,w_j\}]$ is an SR-cycle by definition. So, we my assume, without loss of generality, that elements of $E$ join only vertices of $P_1$ (and thus, that $P_i\sub I$ for $i\neq 1$). However, as $|V_1|>2|P_1|$, this implies that $|I|\geq |V_1\sm P_1|>1$, so this case cannot occur, and thus the desired result holds when $n=1$. Assume, by way of induction, that this result holds for all SR-graphs $(V',E',F')$ satisfying analogous hypotheses, if $(V',F')$ has less than $n$ components. 

Suppose that there is $i\in[n]$ with $\m{H}[W_i]\simeq K_{1,m}$ for some $m\geq2$. Since $|W_i|=|V_i|-|I\cap V_i|$ by definition, and since $|W_i|=m+1$ by assumption, it follows from our hypotheses that
\begin{equation}
m+1>2\mu(\m{H}_i)-|I\cap V_i|\geq 2m - |I\cap V_i|, \label{srg1}
\end{equation}
since $\mu(\m{H}_i)\geq \mu(\m{H}[W_i])=m$. Let $P_1,P_2,\ld,P_k$ be the partite sets of $\m{H}_i$, and let $Q_1=\{w_0\}$ and $Q_2=\{w_1,w_2,\ld,w_m\}$ be the partite sets of $\m{H}[W_i]$; without loss of generality, say $Q_1\sub P_1$ and $Q_2\sub P_2$. Now, since $|V_i|>2\mu(\m{H}_i)$, $k\geq 3$; since $\m{H}[W_i]\simeq K_{1,m}$, this implies that there is $v\in P_3\cap I$. Let $V'$ be obtained from $V$ by replacing $V_i$ with $V_i':=\{w_0,w_1,v\}$, and consider $\m{S}[V']$. Since $\m{H}[V_i']\simeq K_{1,1,1}$, we have $|V_i'|>2\mu(\m{H}[V_i'])$. Moreover, if the vertices in $Q_2\setminus\{w_1\}$
are removed from $V$,
then the number of additional isolated vertices caused by 
the removing of those vertices
is at most $|Q_2\setminus\{w_1\}|$ by Lemma \ref{Add_isol}.
Moreover $|(I\cap V_i)|\geq m$ by (\ref{srg1}),
and so it holds that
$$\begin{array}{lll}
|I(\m{G}[V'])|&\leq |I|
-|(I\cap V_i)\setminus \{v\}|+|Q_2\setminus \{w_1\}|\\
&\leq n-(m-1)+(m-1)=n.\\
\end{array}$$
Therefore, $\m{S}[V']$ still satisfies the hypotheses of our theorem, and clearly, if $\m{S}[V']$ has an SR-cycle then so must $\m{S}$. Moreover, by considering corresponding $W_i'=\{w_0,w_1\}$, we see that $\m{H}[W_i']\simeq K_{1,1}$ (and, in particular, no longer isomorphic to $K_{1,m}$ for any $m\geq2$).  Thus, we may assume that $\m{H}[W_i]\not\simeq K_{1,m}$ (by applying this procedure to any component of $\m{H}$ if necessary). 

Since $\m{H}[W_i]\not\simeq K_{1,m}$ for any $m\geq2$, if $F_i\neq\emptyset$ for all $i\in[n]$ (as this is equivalent to $I(\m{H}[W])=\emptyset$ in this case), then we obtain the desired result by Lemma~\ref{G_lem_kpart}. So, it remains to assume that  $\m{H}[W_i]\not\simeq K_{1,m}$, but that $F_i=\emptyset$ for some $i$. Let $V':=V\sm V_i$ and say $\m{S}[V']=(V',E',F')$. Since the number of components of $(V',F')$ is $n-1$, we may apply our induction hypothesis and prove this result if $|I(\m{G}[V'])|\leq n-1$; we show that this must be the case.  Let $m:=|W_i|$. 
Since $\m{H}_i$ is a complete $k$-partite graph
and $F_i=\emptyset$,
$W_i$ is contained in a partition of $\m{H}_i$, and so $|V_i|>2m$ by assumption;
thus, $|I\cap V_i|=|V_i|-m>m$. Since ${I\cap V'}=I\sm (I\cap V_i)$ and $|I|\leq n$, we have $|{I\cap V'}|\leq n-m-1$. On the other hand, by Lemma \ref{Add_isol}, $|I(\m{G}[V'])|-|{I\cap V'}|\leq m$.
Hence, 
$$m\geq |I(\m{G}[V'])|-|{I\cap V'}|
\geq|I(\m{G}[V'])|-(n-m-1),$$
and thus $|I(\m{G}[V'])|\leq n-1$.
\end{proof}


\section{Proof of Theorem~\ref{MR_TH}}
\label{MF}

Let $G$ be a group and $M$ a subset of $G$.
We denote by $\widetilde{M}$ the symmetric closure of $M$;
$\widetilde{M}=M\cup\{x^{-1}\ |\ x\in M\}$.
For non-empty subsets $M_1,M_2\ldots, M_n$ 
of $G$ consisting of elements not equal to $1$, we say that 
$M_1,M_2,\ldots, M_n$ are mutually reduced in $G$ 
if,
for each finite number of elements 
$g_1,g_2,\ldots, g_m\in\bigcup_{i=1}^{n}\widetilde{M_i}$,
whenever $g_1g_2\cdots g_m=1$, there exists $i\in [m]$ and $j\in[n]$ so that 
$g_i, g_{i+1}\in \widetilde{M_j}$.
If $M_i=\{x_i\}$ for $i\in\{1,\ldots, n\}$
and $M_1,M_2,\ld,M_n$ are mutually reduced,
then we say that $x_1,x_2,\ldots, x_n$ are mutually reduced.

Let $M_1$ and $M_2$ be non-empty subsets
of $G$ consisting of elements not equal to $1$.
If there exist subgroups $A$ and $B$ of $G$
such that 
$M_1\subseteq A$, $M_2\subseteq B$,
and if $AB$ is isomorphic to
the free product $A*B$ of $A$ and $B$,
then $M_1$ and $M_2$ are mutually reduced.
In addition, if $M_1=\{x_1,\ x_2,\
x_1^{-1}x_2\}$
and $M_2=\{y_1,\ y_2,\
y_1^{-1}y_2 \}$
are mutually reduced,
then two elements $x_1y_1^{-1}$ and $x_2y_2^{-1}$
freely generate a free subgroup.
In general, we have the following:

\begin{remark}\label{FSG}
Let $x_i$ and $y_i$, for $i\in I$, be distinct non-identity elements in $G$;
let $M_1:=\{x_i,\ x_i^{-1}x_j\ |\ i,j\in I,\ i\ne j\}$
and $M_2:=\{y_i,\ y_i^{-1}y_j\ |\ i,j\in I,\ i\ne j\}$.
If $M_1$ and $M_2$
are mutually reduced,
then 
$Z=\{z_{i},\ |\ i\in I \}$
is a set of free generators of the subgroup of $G$
generated by $Z$,
where $z_{i}=x_iy_i^{-1}$.
\end {remark}

For a subset $M$ of $G$
and element $x\in G$,
we denote by $M^x$ the set $\{x^{-1}fx \ |\ f\in M\}$.
Then, $(\ast)$ stated in the introduction
can be restated as follows:
\vskip3pt

\begin{tabularx}{12cm}{lX}
$(\ast)$
&For each subset $M$ of $G$ consisting of 
finite number of elements not equal to $1$,
there exist distinct $x_1, x_2, x_3\in G$
such that
$M^{x_i}$ is mutually reduced for each $i\in[3]$.\\
\end{tabularx}
\vskip3pt

In this section,
we will prove Theorem~\ref{MR_TH} 
after first providing four lemmas.
The first of these lemmas is a method established by Formanek \cite{For}
to prove the primitivity of group rings of free products.
We call the method,
which is based on the construction of comaximal ideals,
Formanek's method.
The second one is a basic result on primitive group rings
due to Passman \cite{Ps73}. The other two lemmas are our own, and proofs for them will be provided after their respective statements. 

\begin{lemma}\label{Ps}\mbox{{\rm (\cite[Theorem 2]{Ps73})}}
Let $K'$ be a field and $G$ be group.
If $\Delta(G)$ is trivial and $K'G$ is primitive,
then for any field extension $K$ of $K'$,
$KG$ is primitive.
\end{lemma}

In what follows,
for the pair $v=(f,g)$ of elements $f$ and $g$ in $G$,
we denote the product $fg$ of $f$ and $g$ by $\tilde{v}$. 

\begin{lemma}\label{M_1}
Let $G$ be a non-trivial group, $m>0$, and $n>0$. 
For any distinct, non-trivial elements $f_{ij}$ in $G$, 
over $i\in[3]$ and $j\in[m]$,
and for distinct elements $g_i$ in $G$ 
over $i\in [n]$, we let $S_i:=\{f_{ij}\ |\ j\in [m]\}$ and set
$$\begin{array}{llll}
S&:=\bigcup_{i=1}^{3}S_i,\\
T&:=\{g_i\ |\ i\in[n]\},\\ 
V&:=S\times T,\\
M_i&:=\{f,\ f^{-1}f'\ 
|\  f,f'\in S_i \mbox{ with } f\ne f'\}\ (i\in[3]),\\
I&:=\{v\in V\ |\ \tilde{v}\ne \tilde{w}\
\mbox{ for any } w\in V \mbox{ with } w\ne v \}.
\end{array}$$
Then, if $M_1$, $M_2$ and $M_3$ are mutually reduced,
we have $|I|>n$.
\end{lemma}

\begin{proof}
Suppose, to the contrary, that $|I|\leq n$. We regard $V$ as a vertex set,
and define two edge sets:
$$\begin{array}{lll}
E :=\{ vw\ |\ v, w \in V,\
v\ne w \mbox{ and }, \tilde{v}=\tilde{w}\},\\
F :=\{vw\ |\ v\in S_i\times \{g\},\ w\in S_j\times \{g\}
\mbox{ with } i\ne j
\mbox{ for some } g\in T \}.\\
\end{array}$$
In order to utilize our work in Section~\ref{SRgraph}, let $\m{G}:=(V,E)$, $\m{H}:=(V,F)$, and $\m{S}:=(V,E,F)$; we begin by proving the following claim:

\begin{claim}\label{claimM1}
$\m{S}$ is an SR-graph which contains an SR-cycle. 
\end{claim}

\begin{proof}[Proof of Claim~\ref{claimM1}]
Let us begin by showing that $\m{S}$ is an SR-graph. Since each component of $\m{G}$ is clearly complete by definition, to see that $\m{S}$ is an SR-graph, we need only argue that $E\cap F=\emptyset$.  To see this, assume that $vw\in F$. Then, for $i\ne j$, $v=(f,g)\in S_i\times \{g\}$ and  $w=(h,g)\in S_j\times \{g\}$ for some $g\in T$; since $M_i$ and $M_j$ are mutually reduced and, in particular, $f^{-1}h\ne 1$,we have that $\tilde{v}\ne\tilde{w}$. Hence, $vw\notin E$, and so $\m{S}$ is an SR-graph. 

It remains to argue that $\m{S}$ must contain an SR-cycle. We do this by showing that $\m{S}$, $\m{G}$, and $\m{H}$ satisfy the hypotheses of Theorem~\ref{thm:multipart}.  To that end, let $V_g:=S\times\{g\}$ over $g\in T$. We first notice that $C(\m{H})=\{\m{H}[V_g]\mid g\in T\}$; in particular, $c(\m{H})=n$. Next, we notice that, by the definition of $F$, each $\m{H}[V_g]$ is a complete multipartite graph with exactly three $m$-vertex parts; in particular, over all $g\in T$, $\m{H}_g$ is complete multi-partite and $|V_g|=3m>2m=2\mu(\m{H}_g)$. It remains to show that $|I(\m{G})|\leq n$; as $I=I(\m{G})$ by definition, this follows from our assumption that $|I|\leq n$. 
\end{proof}

By Claim~\ref{claimM1}, there is SR-cycle in $\m{S}$, say $\m{C}=(V_{\m{C}},E_{\m{C}})$ with  $E_{\m{C}}=\{e_1,e_1^{*},\ldots, e_s,e_s^{*}\}$ in $\Sy$ such that $e_t=v_tw_t\in E$ and $e_t^{*}=w_tv_{t+1}\in F$ over $t\in[s]$, and $v_{s+1}=v_1$. Let $v_t=(f_{v_t},g_{v_t})\in S_{i_t}\times \{g_{v_t}\}$ and $w_t=(f_{w_t},g_{w_t})\in S_{j_t}\times g_{w_t}$.
By the definition of $E$, $v_tw_t\in E$ implies that $f_{v_t}g_{v_t}=f_{w_t}g_{w_t}$, and therefore, if $f_{v_t}=f_{w_t}$ then $g_{v_t}=g_{w_t}$, which contradicts the fact $v_t\ne w_t$. Hence we have that
\begin{equation}\label{R_C_3}
f_{v_t}\ne f_{w_t}.
\end{equation}
On the other hand,
$w_tv_{t+1}\in F$ implies that
$g_{w_t}=g_{v_{t+1}}$ and
\begin{equation}\label{R_C_4}
j_t\ne i_{t+1}.
\end{equation}
Since $g_{w_t}=g_{v_{t+1}}$, it follows that
$$\begin{array}{lllll}
f_{v_1}g_{v_1}=&f_{w_1}g_{w_1},&\\
                     &f_{v_2}g_{w_1}=&f_{w_2}g_{w_2},\\
&&\hskip1cm\ddots\\
&&f_{v_s}g_{w_{s-1}}=f_{w_s}g_{w_s},\\
\end{array}
$$
and $g_{w_s}=g_{v_1}$. Solving these equations yields 
$$f_{v_1}^{-1}f_{w_1}f_{v_2}^{-1}f_{w_2}
\cdots f_{v_s}^{-1}f_{w_s}=1.$$
However, since $f_{v_{t}}^{-1}\in M_{i_{t}}$
and $f_{w_{t}}\in M_{j_{t}}$,
if $i_{t}=j_{t}$ then 
$f_{v_{t}}^{-1}f_{w_{t}}\in M_{j_{t}}$
with $f_{v_{t}}^{-1}f_{w_{t}}\ne 1$ by (\ref{R_C_3}).
Moreover, $f_{v_{t+1}}^{-1}f_{w_{t+1}}\in M_{i_{t+1}}$ with $j_t\ne i_{t+1}$
by (\ref{R_C_4}),
which contradicts the hypothesis that
$M_{1}$, $M_{2}$ and $M_{3}$
are mutually reduced. Thus, we have reached the desired contradiction, and our proof is complete. 
\end{proof}

\begin{lemma}\label{M_2}
Let $G$ be a non-trivial group and $n>0$.
For each $i\in[n]$, 
let $f_{i1},\ldots, f_{im_i}$ be distinct elements of $G$,
$f_{ip}\ne f_{iq}$ for $p\ne q$, and let
$x_{il}$, $i\in[n]$ and $l\in[3]$,
be distinct elements in $G$. We set
$$\begin{array}{llll}
S&:=\bigcup_{i=1}^{n}S_i,
\mbox{ where } S_i:=\{f_{ij}\ |\ j\in[m_i]\},\\
X&:=\bigcup_{i=1}^{n}X_i,
\mbox{ where }X_i=\{x_{il}\ |\ l\in[3]\},\\
V&:=\bigcup_{i=1}^{n}V_i,
\mbox{ where }V_i=X_i\times S_i,\\
I&:=\{v\in V\ |\ \tilde{v}\ne \tilde{w}\
\mbox{ for any } w\in V \mbox{ with } w\ne v \}.
\end{array}$$
If $x_{ij}$ are mutually reduced elements over $i\in[n]$ and $j\in[m_i]$,
then $|I|>m$,
where $m:=m_1+\cdots+m_n$.
\end{lemma}

\begin{proof}
Suppose, to the contrary, that $|I|\leq m$. We regard $V$ as a vertex set,
and set 
$$\begin{array}{lll}
&E :=\{ vw\ |\ v, w \in V,\
v\ne w \mbox{ and } \tilde{v}=\tilde{w}\},\\
&F :=\{vw\ |\ v, w\in V_i(f),\ v\ne w,
\mbox{ for some } f\in S_i
\mbox{ and }
i\in[n] \},
\end{array}$$
where $V_i(f)=X_i\times \{f\}$ for $i\in[n]$ and $f\in S_i$.
Note that $|V|=3nm$; in particular, 
$(x_{il},f_{ij})=(x_{ps},f_{pq})$
if and only if $(i,l,j)=(p,s,q)$.
In order to utilize our work in Section~\ref{SRgraph}, let $\m{G}:=(V,E)$, $\m{H}:=(V,F)$, and $\m{S}:=(V,E,F)$; we begin by proving the following claim:

\begin{claim}\label{claimM2}
$\m{S}$ is an SR-graph which contains an SR-cycle. 
\end{claim}

\begin{proof}[Proof of Claim~\ref{claimM2}]
Following similar arguments to those used in the first paragraph of the proof of Claim~\ref{claimM1}, we can see that $\m{S}$ is an SR-graph; so, our task is to show that $\m{S}$ has an SR-cycle. By definition of $E$, we have $I=I(\m{G})$; by definition of $F$, we have $C(\m{H})=\{\HGh[V_i(f)]\mid f\in S_i,\,i\in[n]\}$. Moreover, each $\HGh[V_i(f)]\in C(\m{H})$ is clearly a complete $3$-vertex graph by definition, and in particular, $c(\m{H})=|V|/3$. 
Thus, our proof is complete by Theorem~\ref{thm:complete} if $c(\m{G})+c(\m{H})<|V|+1$; as $c(\m{H})=|V|/3$, this holds if
\begin{equation}
c(\m{G})<\frac{2}{3}|V|+1. \label{asumsim}
\end{equation}

Because there exists a connected component satisfying (\ref{asumsim}) whenever $\m{S}$ satisfies (\ref{asumsim}).

Now, since $|I|\leq m=c(\m{H})=|V|/3$ by assumption, $\m{G}$ can have at most $|V|/3$ one-vertex components (while all other components of $\m{G}$ have at least two vertices); thus, since the components of $\m{G}$ partition $\m{S}$, $c(\m{G})\leq |V|/3 + (1/2)(2|V|/3)=2|V|/3$, and so (\ref{asumsim}) holds. 
\end{proof}

By Claim~\ref{claimM2}, there is SR-cycle in $\m{S}$, say $\m{C}=(V_{\m{C}},E_{\m{C}})$ with  $E_{\m{C}}=\{e_1,e_1^{*},\ldots, e_s,e_s^{*}\}$ such that $e_t=v_tw_t\in E$, $e_t^{*}=w_tv_{t+1}\in F$, for $t\in[s]$, and $v_{s+1}=v_1$. Let $v_t=(x_{t},f_{t})\in V_{i_t}(f_{t})$ with $f_{t}\in S_{i_t}$ and $w_t=(y_{t},g_{t})\in V_{j_t}(g_{t})$ with $g_{t}\in S_{j_t}$.
By the definition of $E$, $v_tw_t\in E$ implies that $v_t\ne w_t$ and $x_{t}f_{t}=y_{t}g_{t}$,
and so $x_{t}\ne y_{t}$. In addition, $w_tv_{t+1}\in F$ implies that
$j_{t}=i_{t+1}$, $g_t=f_{t+1}$ and $y_{t}\ne x_{t+1}$. Hence, 
$$\begin{array}{lllll}
x_{1}f_{1}=&y_{1}g_{1},&\\
                     &x_{2}g_{1}=&y_{2}g_{2},\\
&&\hskip1cm\ddots\\
&&x_{s}g_{s-1}=y_{s}g_{s}
\quad\mbox{ and }\quad g_{s}=f_{1},\\
\end{array}
$$
where $x_{t}\ne y_{t}\ne x_{t+1}$.
Eliminating $f_1$ and $g_{t}$'s in the above equations,
we get
$$x_{1}^{-1}y_{1}x_{2}^{-1}y_{2}
\cdots x_{s}^{-1}y_{s}=1.$$
But this contradicts the hypothesis that
$x_{i}$'s and $y_{i}$'s are mutually reduced.
\end{proof}

With these lemmas in place, we are now in the position to prove Theorem \ref{MR_TH}.

\begin{proof}[Proof of Theorem~\ref{MR_TH}]
Let $B$ be the basis of a non-abelian free subgroup of $G$ whose cardinality is the same as that of $G$.
If $|G|>\aleph_0$ then $|G|=|B|$,
and in addition,
a two generator free group always contains a free subgroup
generated by infinitely many generators.
Thereby,
we may assume that the cardinality of $B$ is also the same as $G$.
In addition,
since $|R|\leq |G|$, we have that $|B|=|RG|$.
We can divide $B$ into three subsets $B_1$, $B_2$ and $B_3$
each of whose cardinality is $|B|$.
It is then obvious that
the elements in $B$
are mutually reduced.
Let $\varphi$ be a bijection from $B$ to $RG\setminus \{0\}$
and $\sigma_s$ a bijection from $B$ to $B_s$, $s\in[3]$.
For $b\in B$, we denote $\sigma_s(b)$ by $b_s$.

For $b\in B$,
let $\varphi(b)=\sum_{f\in F_b}\alpha_{f}f$,
where $\alpha_f\in R$
and $F_b=Supp(\varphi(b))$ is the support of $\varphi(b)$.
We set
$$M_{b}=\{ f^{\pm 1},\ f^{-1}f^{\prime}\
|\  f,f^{\prime}\in F_b, f\ne f^{\prime}\}.$$
As $G$ satisfies $(\ast)$,
there are $x_{b1}, x_{b2}, x_{b3}\in G$
with 
$$M^{x_{bt}}_{b}=\{ x_{bt}^{-1}f^{\pm 1}x_{bt},\
x_{bt}^{-1}f^{-1}f^{\prime}x_{bt}\ |\
f, f^{\prime}\in F_b,\ f\ne f^{\prime}\}\quad (t\in[3])$$
are mutually reduced.
We next define $\varepsilon(b)$ and $\varepsilon^{1}(b)$ by
\begin{equation}\label{MRP_1}
\varepsilon(b)=\sum_{s=1}^{3}\sum_{t=1}^{3}
b_sx_{bt}^{-1}\varphi(b)x_{bt}\
\mbox{ and }\
\varepsilon^{1}(b)=\varepsilon(b)+1.
\end{equation}
Note that $\varepsilon(b)$ is an element
in the ideal of $RG$ generated by $\varphi(b)$.
Let $\rho=\sum_{b\in B}\varepsilon^{1}(b)RG$ be the right ideal 
generated by $\varepsilon^{1}(b)$ for all $b\in B$.
If $w\in\rho$,
then 
we can express $w$ by
\begin{equation}\label{MRP_2}
w=\sum_{b\in A}\varepsilon^{1}(b)u_b
=w_1+w_2,
\mbox{ where } w_1=\sum_{b\in A}\varepsilon(b)u_b
\mbox { and } w_2=\sum_{b\in A}u_b,
\end{equation}
for some non-empty finite subset $A$ of $B$
and $u_b$ in $RG$.
According to Formanek's method; Lemma \ref{Fmethod},
in order to prove that $RG$ is primitive,
we need only to show that $\rho$ is proper; $\rho\ne RG$.
To do this,
it suffices to show that $w\ne 1$.

Let 
$u_b=\sum_{h\in H_b}\beta_{h}h$,
where $H_b=Supp(u_b)$ and $\beta_{h}\in R$.
Substituting $\varphi(b)=\sum_{f\in F_b}\alpha_{f}f$ 
into (\ref{MRP_1}), we obtain the following expression
of $\varepsilon(b)u_b$:
\begin{equation}\label{MRP_3}
\varepsilon(b)u_b
=\sum_{s=1}^{3}b_sE_{b},
\mbox{ where } E_{b}=\sum_{t=1}^{3}\sum_{f\in F_b}\sum_{h\in H_b}
\alpha_{f}\beta_{h}x_{bt}^{-1}fx_{bt}h.
\end{equation}
We can see that there exist
more than $|H_b|$ isolated elements
in the expression (\ref{MRP_3}) of $E_{b}$;
that is,
$m_b>|H_b|$,
where $m_b=|Supp(E_{b})|$.
In fact, 
let $X_b=\{x_{b1}, x_{b2}, x_{b3}\}$,
$\Gamma_b=X_b\times F_b\times H_b$, and
$$I_b=\{c\in \Gamma_b\ |\ \tilde{c}\ne \tilde{c'}
\mbox{ for any } c'\in\Gamma_b
\mbox{ with } c'\ne c\},$$
where $\tilde{c}=x_{bt}^{-1}fx_{bt}h$ for $c=(x_{bt},f,h)$.
Since $M^{x_{bt}}_{b}$
$(t\in[3])$ are mutually reduced,
taking $I_b$ as $I$, $H_b$ as $T$ and $\Gamma_b$ as $V$ in lemma \ref{M_1},
it follows from lemma \ref{M_1}
that $|I_b|>|H_b|$ and thus $m_b>|H_b|$
because of $m_b\geq |I_b|$, as desired.
Now,
since $b_s$ $(b\in A, 1\leq s\leq 3)$
are mutually reduced,
by Lemma \ref{M_2},
taking $|A|$ as $n$, $Supp(E_{b})$ as $S_i$, 
and $\{b_1,b_2,b_3\}$ as $X_i$ in Lemma \ref{M_2},
we have 
$|Supp(w_1)|
>\sum_{b\in A}m_b$.
Hence we have that
\begin{equation*}
|Supp(w)|\geq |Supp(w_1)|-|Supp(w_2)|>\sum_{b\in A}m_b-\sum_{b\in A}|H_b|>0,
\end{equation*}
which implies $|Supp(w)|\geq 2$. In particular, $w\ne 1$. Thus, $RG$ is primitive.

Finally,
we shall show that $KG$ is primitive for any field $K$.
Let $K^{\prime}$ be a prime field.
Since $G$ satisfies $(\ast)$ and $|K^{\prime}|\leq |G|$,
we have already seen that $K^{\prime}G$ is primitive.
By Lemma \ref{Ps}, it suffices to show that $\Delta(G)=1$. Let $g$ be a nonidentity element in $G$.
We can see that there exist infinitely many conjugate elements of $g$.
In fact,
if this is not the case,
then the set $M$ of conjugate elements of $g$ in $G$
is a finite set.
Since $G$ satisfies $(\ast)$, for $M$,
there exists $x_1, x_2\in G$ such that
$M^{x_1}$ and $M^{x_2}$ are mutually reduced.
Since $g$ is in $M$,
$(x_1^{-1}gx_1)(x_2^{-1}fx_2)^{-1}\ne 1$ for any $f\in M$,
and thus $x_1^{-1}gx_1\ne x_2^{-1}fx_2$.
Hence $(x_1x_2^{-1})^{-1}g(x_1x_2^{-1})\ne f$ for any $f\in M$,
which implies $(x_1x_2^{-1})^{-1}g(x_1x_2^{-1})\not\in M$, 
a contradiction.
\qed
\vskip12pt

Let $G$ be a countably infinite group
and  $g_1, g_2\in G$ with $g_1\ne g_2$.
If $G$ satisfies $(\ast)$,
then for $M=\{g_i^{\pm 1}, g_i^{-1}g_j\ |\ i,j=1,2, i\ne j\}$,
there exist $x_1, x_2\in G$ such that $M^{x_1}$ and $M^{x_2}$
are mutually reduced. By Remark \ref{FSG},
$\langle z_1, z_2\rangle$ is a free subgroup of $G$,
where $z_i=x_1^{-1}g_ix_1x_2^{-1}g_i^{-1}x_2$.
Hence in Theorem \ref{MR_TH},
the assumption on existence of a free subgroup 
is not needed in the case of $|G|=\aleph_0$.
\end{proof}


\section{HNN extensions and amalgamated free products}
\label{PAFP}

In this section, we use Theorem \ref{MR_TH} to establish results concerning the primitivity of group algebras
of HNN extensions and amalgamated free products;  we extend results from
\cite{Bal}, \cite{For}, \cite{Ni07}, and \cite{Ni11}. It is easy to see that a non-abelian free group, and more generally,
a non-abelian locally free group,
satisfies $(\ast)$.
In fact, if $h_1,\ldots,h_m$ are elements of a locally free group $H$ for any $m\in\N$,
then they lie in a free subgroup $\langle X\rangle$ of $H$
generated by a base set $X$ with $|X|>1$.
For $x, y\in X$ with $x\ne y$,
let $x_i:=x^{2p+i}yx^{2p+i}$ $(i\in[3])$
where $p$ is the maximum number of all the lengths of $h_j$s, over $j\in[m]$, with respect to $X$. We see that the associated $M^{x_i}$s are mutually reduced. Using Theorem \ref{MR_TH},
we can reprove the main theorem in \cite{Ni11}:
$KH$ is primitive for any field $K$ provided
$H$ has a free subgroup 
whose cardinality is the same as that of $H$.

Let $G$ be a group. For subgroups $A$ and $B$ of $G$,
let $G^{*}=\langle G, t\ |\ t^{-1}at=\varphi(a),\ a\in A\rangle$
be an HNN extension with base $G$ and a stable letter $t$,
where $\varphi:A\rightarrow B$ is an isomorphism. For $g_0,\ldots g_n\in G$ and $\varepsilon_i=\pm 1$, where $n\in\N_0$ and $i\in[n]$,
a sequence $g_0, t^{\varepsilon_1}, \ldots, t^{\varepsilon_n}, g_n$
is said to be reduced if there is no consecutive subsequence
$t^{-1}, g_i, t$ with $g_i\in A$
or $t, g_i, t^{-1}$ with $g_i\in B$.
For $u=g_0t^{\varepsilon_1}\cdots t^{\varepsilon_n}g_n\in G^{*}$,
if $g_0, t^{\varepsilon_1}, \ldots t^{\varepsilon_n}, g_n$
is reduced, then we say that the product
$g_0t^{\varepsilon_1}\cdots t^{\varepsilon_n}g_n$ 
is a reduced form of $u$.
By the normal form theorem for HNN extensions,
if $u=g_0t^{\varepsilon_1}\cdots t^{\varepsilon_n}g_n=1$,
then either $n=0$ and $g_0=1$
or $n\geq 1$ and $u$ is not reduced.
Moreover, if $u\in G^{*}$,
then $u$ is always uniquely expressed by
the normal form, which is a reduced form,
as follows:
$$u=g_0t^{\varepsilon_1}
\cdots t^{\varepsilon_n}g_n,$$
where
(i)\
$g_0$ is arbitrary element in $G$,
(ii)\
if $\varepsilon_i=-1$, then $g_i$ 
is representative of a right coset of $A$ in $G$,
(iii)\
if $\varepsilon_i=+1$, then $g_i$ 
is representative of a right coset of $B$ in $G$,
and (iv)\
there is no consecutive subsequence $t^{-1}1t$ or $t1t^{-1}$.
In the above, as usual, $1$ is the representative of both $A$ and $B$.
In what follows, for $l_1,\ldots, l_n\in\Zh$,
whenever we say that $u=g_0t^{l_1}\cdots t^{l_n}g_n\ne 1$
is the normal form (resp. a reduced form) of $u$,
it means that it is the normal form (resp. a reduced form),
and also that
$$\begin{array}{lll}
l_i\ne 0 \mbox{ for each } i\in[n] & \mbox{if } n>0,\\
g_0\ne 1 & \mbox{if } n=0,\\
g_i\ne 1 \mbox{ for } 0<i<n & \mbox{if } n>1.\\
\end{array}$$

If $A=G$, then $G^{*}$ is said to be an ascending HNN extension of $G$.
In this case,
$G^{*}$ is isomorphic to the cyclic extension of $G_{\infty}$,
where $G_{\infty}=\cup_{i=1}^{\infty}t^iGt^{-i}$.
In addition,
if $B\subsetneq G$ then $G^{*}$ is called a strictly ascending
HNN extension of $G$.
In \cite{Ni07}, one of the present authors proved
that $KG$ is primitive for any field $K$,
provided that $G^{*}$ is a strictly ascending
HNN extension of a non-abelian free group $G$.
We can generalize this result as follows:

\begin{theorem}\label{HNN_TH}
Let $G$ be a group.
For nontrivial subgroups $A$ and $B$ of $G$,
let $G^{*}=\langle G, t\ |\ t^{-1}at=\varphi(a),\ a\in A\rangle$
be an HNN extension with base $G$ and a stable letter $t$,
where $\varphi:A\rightarrow B$ is an isomorphism.
\begin{enumerate}[(1)]
\item If $A\cup B\subsetneq G$ and
there exists $g\in G$ such that
either $g^{-1}Ag\cap A=1$ or $g^{-1}Bg\cap B=1$,
then $KG^{*}$ is primitive for any field $K$. \label{thmfouroneone}
\item Suppose that $G^{*}$ 
has a free subgroup whose cardinality is the same as that of $G$. 
If $A=G$, $B\subsetneq G$, and $G$ satisfies $(\ast)$,
then $KG^{*}$ is primitive for any field $K$. \label{thmfouronetwo}
\end{enumerate}
\end{theorem}

The next basic result on group rings
is needed in the proof below.
We refer the reader to Passman \cite{Ps77}
for a detailed discussion of this topic.

\begin{lemma}\label{PBR}\mbox{{\rm (\cite[Theorem~1]{Zal})}}
Let $K$ be a field, $G$ a group, and $N$ a normal subgroup of $G$ with $\Delta(G)=1$ and $\Delta(G/N)=G/N$. If $KN$ is primitive, then so is $KG$. \\
\end{lemma}
\vspace{-0.25in}

\begin{proof}[Proof of Theorem~\ref{HNN_TH}]
We begin by proving (\ref{thmfouroneone}). By Theorem~\ref{MR_TH},
it suffices to show that $G^{*}$ satisfies $(\ast)$
and has a free subgroup
whose cardinality is the same as that of $G^{*}$.
Replacing $\varphi$ with $\varphi^{-1}$ if necessary,
we assume that there exists $g\in G$ with 
$g^{-1}Ag\cap A=1$.

We shall first show that $G^{*}$ satisfies $(\ast)$.
Let $M$ be a set of finitely many non-trivial elements in $G$.
For $u\in M$, let
\begin{equation}\label{HNN1}
u=u_{0}t^{l(u1)}\cdots t^{l(un_u)}u_{n}
\end{equation}
be the normal form of $u$.
Choose $q\in\Z$ so that 
$q>\sum_{j=1}^{n_u}|l(uj)|$ for any $u\in M$,
and put $x_i=t^{-q_i}gth^{-1}t^{q_i}$, for $i\in[3]$,
where $h\in G\setminus (A\cup B)$ and $q_i=q+i$;
we will show that the $M^{x_i}=\{x_i^{-1}ux_i\ |\ u\in M\}$
are mutually reduced, implying that $G^{*}$ satisfies $(\ast)$.
It suffices to show that,
for each $v_1,\ldots, v_k\in\cup_{i=1}^{3}M^{x_i}$
with $\{v_j, v_{j+1}\}\not\subseteq M^{x_i}$,
there are $l_1,\ldots l_n\in\Zh$ and $g_0,\ldots. g_n\in G$
so that $w=g_0t^{l_1}\cdots t^{l_n}g_n$ is reduced and 
\begin{equation}\label{HNN2}
t^{-q_i}ht^{-1}wth^{-1}t^{q_j}
\end{equation}
is a reduced form of $v_1\cdots v_k$,
where $v_1\in M^{x_i}$ and $v_k\in M^{x_j}$.
Assume first that $k=1$,
and then $v_1=x_i^{-1}ux_i$.
Let $u_{0}t^{l_1}\cdots t^{l_n}u_{n}\ne 1$ 
be the normal form of $u$. Then,
\begin{equation}\label{HNN3}
v_1=t^{-q_i}ht^{-1}g^{-1}t^{q_i}ut^{-q_i}gth^{-1}t^{q_i},
\mbox{ where } u=u_{0}t^{l_1}
\cdots t^{l_n}u_{n}.
\end{equation}
If either $u_0\not\in B$ or $l_1>0$,
then the expression of $v_1$ in (\ref{HNN3})
is a reduced form.
We may assume therefore that $u_0\in B$ and $l_1\leq 0$.
Since $tu_0t^{-1}=\varphi^{-1}(u_0)$,
we have
\begin{equation}\label{HNN4}
t^{q_i}ut^{-q_i}=\left\{
\begin{array}{llll}
t^{q_i-1}\varphi^{-1}(u_0)t^{-q_i+1} & \mbox{if } l_1=0\
 (\mbox{i.e., } n=0),\\
t^{q_i-1}\varphi^{-1}(u_0)u_1t^{l_2}
\cdots t^{l_n}u_{n}t^{-q_i} & \mbox{if } l_1=-1\\
t^{q_i-1}\varphi^{-1}(u_0)t^{l_1+1}
\cdots t^{l_n}u_{n}t^{-q_i} & \mbox{if } l_1<-1.\\
\end{array}\right.
\end{equation}
If $l_1=0$ and $\varphi^{-1}(u_0)\not\in B$,
then $t^{q_i-1}\varphi^{-1}(u_0)t^{-q_i+1}$
is a reduced form of $t^{q_i}ut^{-q_i}$, because $q_i>1$.
Similarly, if either $l_1=-1$ and $\varphi^{-1}(u_0)u_1\not\in B$
or $l_1<-1$ and $\varphi^{-1}(u_0)\not\in B$,
then the expressions in (\ref{HNN4})
are respectively reduced.
Substituting these for $t^{q_i}ut^{-q_i}$ in (\ref{HNN3}),
$v_1$ has a reduced form as in (\ref{HNN2}) for each case.
We may assume therefore that
$$\left\{\begin{array}{lll}
\varphi^{-1}(u_0)\in B & \mbox{if } l_1=0 \mbox{ or } l_1<-1,\\
\varphi^{-1}(u_0)u_1\in B & \mbox{if } l_1=-1.\\
\end{array}\right.$$
Note that if $l_1=0$ then $u_0\ne 1$, 
and also that if $l_1=-1$ then $\varphi^{-1}(u_0)u_1\ne 1$
because $1\ne u_1\not\in A$ and $\varphi^{-1}(u_0)\in A$.
Since $q_i>\sum_{j=1}^{n}|l_j|+i$,
we can proceed with this procedure for (\ref{HNN4}) under necessary assumption
until we get
$$t^{q_i}ut^{-q_i}=\left\{
\begin{array}{llll}
\varphi^{-q_i}(u_0) & \mbox{if } l_1=0\
 (\mbox{i.e., } n=0),\\
\varphi^{-q_i-(l_1+\cdots+l_n)}(a_nu_n)
 & \mbox{if } l_i<0\ (i\in[n]),\\
\end{array}\right.
$$
where $a_1=\varphi^{l_1}(u_0)$ and $a_{i+1}=\varphi^{l_{i+1}}(a_iu_i)$
for $i\in[n-1]$.
Since both $\varphi^{-q_i}(u_0)$ 
and $\varphi^{-q_i-l}(a_nu_n)$ are non-trivial and in $A$,
where $l=l_1+\cdots+l_n$,
we see that $g^{-1}\varphi^{-q_i}(u_0)g\not\in A$
and $g^{-1}\varphi^{-q_i-l}(a_nu_n)\not\in A$.
This implies that
$$
v_1=\left\{
\begin{array}{llll}
t^{-q_i}ht^{-1}g^{-1}\varphi^{-q_i}(u_0)gth^{-1}t^{q_i}
& \mbox{if }  n=0,\\
t^{-q_i}ht^{-1}g^{-1}\varphi^{-q_i-l}(a_nu_n)t^lgth^{-1}t^{q_i}
& \mbox{if }  n>0\\
\end{array}\right.
$$
are respectively reduced forms of $v_1$.

Now, let $t^{-q_i}ht^{-1}wth^{-1}t^{q_j}$
and $t^{-q_r}ht^{-1}
w^{\prime}th^{-1}t^{q_s}$
are reduced forms of $v$ and $v'$ as in (\ref{HNN2}), respectively.
If $j\ne r$ then $vv'$ also has a reduced form as in (\ref{HNN2}),
and so it can be easily seen by
induction on $k$ that for each $v_1,\ldots, v_k\in\cup_{i=1}^{3}M^{x_i}$
with $\{v_j, v_{j+1}\}\not\subseteq M^{x_i}$,
$v_1\cdots v_k$ has a reduced form as in (\ref{HNN2}). Thus, $G^{*}$ satisfies $(\ast)$.

It remains to prove that $G^{*}$ has a free subgroup
whose cardinality is the same as that of $G^{*}$.
We may asume that $|G^{*}|>\aleph_0$.
Recall that $g\in G\setminus A$ with $g^{-1}Ag\cap A=1$.
If $|A|=|G^{*}|$
then we set
\begin{equation}\label{HNN5}
\begin{array}{ll}
&M_1=\{x_a^{\pm 1},\ x_a^{-1}x_{a'}\ |\ a, a'\in A\setminus\{1\},\ a\ne a'\}\\
\mbox{and}
&M_2=\{y_a^{\pm 1},\ y_a^{-1}y_{a'}\ |\ a, a'\in A\setminus\{1\},\ a\ne a'\},\\
\mbox{where}
&x_a=t^{-1}agt \mbox{ and } y_a=t^{-2}agt^2.\\
\end{array}
\end{equation}
Since $ag\not\in A$ and $g^{-1}a^{-1}a'g\not\in A$,
we have that $x_a\ne x_{a'}$ and $y_a\ne y_{a'}$ for $a\ne a'$.
In paticular, $|M_1|=|M_2|=|G^{*}|$.
Moreover, it is obvious that $M_1$ and $M_2$ are mutually reduced,
and so $Z=\{x_ay_a^{-1},\ |\ a\in  A\setminus\{1\}\}$
generates the free subgroup whose cardinality is the same as that of $G^{*}$
by Remark \ref{FSG}.

Next suppose that $|A|<|G^{*}|$.
Let $S$ be the set consisting
of representatives of a right coset of $A$ in $G$.
We have then that $|S|=|G^{*}|$.
In (\ref{HNN5}),
replacing $A$ with $S$,
$x_a=t^{-1}agt$ with $x_a=t^{-1}at$
and $y_a=t^{-2}agt^2$ with $y_a=t^{-2}at^2$,
since for $a, a'\in S$ with $a\ne a'$,
both $a$ and $a^{-1}a'$ are not in $A$,
we repeat the same argument as in the above,
and get the desired result.

We now prove (\ref{thmfouronetwo}). Let $G_i=t^{i}Gt^{-i}$ and $G_{\infty}=\cup_{i=0}^{\infty}G_i$.
We can easily see that $G_{\infty}$ is a normal subgroup of $G^{*}$,
and also that $G^{*}$ is isomorphic to the cyclic extension of $G_{\infty}$.
In particular, $\Delta(G^{*}/G_{\infty})=G^{*}/G_{\infty}$.
If $M$ is a set of finitely many non-trivial elements in $G_{\infty}$,
then $M\subseteq G_i$ for some $i\in \Nat$.
Since $G_i$ is isomorphic to $G$,
$G_i$ satisfies $(\ast)$.
It follows from Theorem~\ref{MR_TH} that
$KG_{\infty}$ is primitive for any field $K$.
By Lemma~\ref{PBR}, it remains to prove that
$\Delta(G^{*})=1$.

Suppose, to the contrary, that $\Delta(G^{*})\ne 1$.
Let $g$ be in $\Delta(G^{*})$ with $g\ne 1$.
Since $[G^{*}:C_{G^{*}}(g)]<\infty$,
we have $[G:C_{G}(g)]<\infty$.
On the other hand, as we saw at the end of the proof of Theorem~\ref{MR_TH},
$\Delta(G)=1$, which implies $g\not\in G$.
By the normal form theorem,
there exist $n, l\geq 0$ and $f\in G$
such that $g=t^{n}ft^{-l}$,
where $f\not\in B(=\varphi(G))$ if neither $n=0$ nor $l=0$.
Replacing $g$ with $g^{-1}$ if necessary,
we may assume that $n\geq l\geq 0$, and then $f\not\in B$ unless $l=0$.
Since $[G^{*}:C_{G^{*}}(g)]<\infty$,
there exists $m\geq 1$ such that $t^{m}gt^{-m}=g$,
and so $t^{m+n}ft^{-l-m}=t^nft^{-l}$,
which implies $f=\varphi^m(f)\in B$.
Hence we get $l=0$; $g=t^nf$, where $n>0$ and $f\in B$.
Let $h\in G\setminus B$.
Again by $[G^{*}:C_{G^{*}}(g)]<\infty$,
there exists $m\geq 1$ such that $(th)^{m}g(th)^{-m}=g$.
Since $ht^n=t^n\varphi^n(h)$ and $t^{-1}h^{-1}=\varphi(h^{-1})t^{-1}$,
we have that
$$\begin{array}{lll}
(th)^{m}g(th)^{-m}
&=(th)^{m}t^nf(th)^{-m}\\
&=(th)^{m-1}t^{n+1}\varphi^n(h)fh^{-1}\varphi(h^{-1})t^{-2}(th)^{-m+2}\\
&\vdots\hskip2cm \vdots\\
&=t^{m+n}\varphi^{n+m-1}(h)\cdots
\varphi^n(h)fh^{-1}\varphi(h^{-1})\cdots\varphi^{m-1}(h^{-1})t^{-m}\\
&=t^nf,\\
\end{array}
$$
which implies that
$$h^{-1}=f^{-1}\varphi^n(h^{-1})\cdots\varphi^{n+m-1}(h^{-1})
\varphi^m(f)\varphi^{m-1}(h)\cdots\varphi(h).$$
Since $f\in B$,
we get $h^{-1}\in B$, a contradiction.
\end{proof}

For the remainder of this section,
let $A*_{H}B$ be the free product of $A$ and $B$ with $H$ amalgamated,
and suppose that $A\ne H\ne B$.
For $x\in A*_{H}B$ with $x\not\in H$
and for $u_i\in (A\cup B)\setminus H$ $(i\in[n])$,
$x=u_1\cdots u_n$ is a normal form for $x$
provided $u_i$ and $u_{i+1}$ are not both in $A$ or not both in $B$.
Although a normal form $x=u_1\cdots u_n$ is not unique,
the length $n$ of $x$ is well defined
and it is denoted here by $l(x)$.
If $x\in H$, we define $l(x)=0$.
For $x, V_1, \ldots, V_m\in A*_{H}B$,
we write $x\equiv V_1\cdots V_m$ and say that 
the product $V_1\cdots V_m$ is a reduced form if
$x=V_1\cdots V_m$ and $l(x)=l(V_1)+\cdots+l(V_m)$. We consider the following condition on $A*_{H}B$:
\vskip3pt

\begin{tabularx}{12cm}{lX}
$(\dagger)$&
$B\ne H$
and there exist elements $a$ and $a_{*}$ in $A\setminus H$
such that $aa_*\ne 1$ and $a^{-1}Ha\cap H=1$.\\
\end{tabularx}
\vskip3pt

\noindent
It is clear that either $aa_*\not\in H$
or $a_{*}a\not\in H$
provided $a$ and $a_{*}$ are elements as described in $(\dagger)$.
We shall prove the following theorem
which generalizes \cite[Theorem 3.1]{Bal}:

\begin{theorem}\label{PLAFP_TH}
Let $R$ be a domain and $G$ a non-trivial group
which has a free subgroup
whose cardinality is the same as that of $G$.
Suppose that for each $n\in\N$ and $f_1,\ldots,f_n\in G$, there exists a subgroup $N$
containing $f_1,\ldots,f_n$,
such that $N$ is isomorphic to $A*_{H}B$ which satisfies
$(\dagger)$. Then the group ring $RG$ is primitive
provided $|R|\leq |G|$.
In particular,
$KG$ is primitive for any field $K$.
\end{theorem}

If $A\ne H\ne B$,
then $A*_{H}B$ always has a countable free subgroup.
Hence, in Theorem~\ref{PLAFP_TH},
the assumption on existence of a free subgroup 
is needed only if $|G|>\aleph_0$. By  Theorem \ref{MR_TH}, to prove Theorem~\ref{PLAFP_TH}, 
it suffices to show that
$G$ satisfies $(\ast)$.
Since, for each $n\in\N$ and 
$f_1,\ldots,f_n\in G$, there is a subgroup $N=A*_{H}B$
containing $f_1,\ldots,f_n$ such that
$N$ satisfies $(\dagger)$ by assumption, we need only show that
if $A*_{H}B$ satisfies
$(\dagger)$,
then $A*_{H}B$ satisfies $(\ast)$.
In fact, if $b\in B\setminus H$ and $a, a_{*}\in A$
which satisfy the conditions $aa_{*}\ne 1$ and $a^{-1}Ha\cap H=1$,
then for $i\in[3]$,
\begin{eqnarray}
x_i =& (b^{-1}a)^{\omega_i}a_{*}b^{-1}a_{*}^{-1}(b^{-1}a)^{\omega_i}
&\mbox{ if } \ aa_{*}\not\in H \label{xi1}\\
x_i =& (b^{-1}a^{-1})^{\omega_i}a_{*}^{-1}b^{-1}a_{*}(b^{-1}a^{-1})^{\omega_i}
&\mbox{ if } \ a_{*}a\not\in H \label{xi2}
\end{eqnarray}
are desired elements in $A*_{H}B$,
where $\omega_i=l+i$ and $l$ is the maximum number in the set
$\{ l(f_i)\ |\
1\leq i\leq n\}$.
That is, for $M=\{ f_1,\ldots,f_n\}$,
$M^{x_i}$ $(i=1,2,3)$ are mutually reduced.
We shall confirm this after preparing a lemma.

\begin{lemma}\label{PAFP_L1}
Let $G=A*_{H}B$. Suppose that $G$ satisfies $(\dagger)$,
and let $a$ be an element as in $(\dagger)$ above.
Let $1\ne f\in G$ with $l(f)=l$
and $W=(a^{-1}b)^mf(b^{-1}a)^m$,
where $m$ is a positive integer and $b\in B\setminus H$.

If $m>l+1$,
then a reduced form of $W$ has the form
\begin{equation}\label{redform}
W\equiv (a^{-1}b)V(b^{-1}a) 
\mbox{ for some non-empty word }
V,
\end{equation}
otherwise $W=(b^{-1}a)^{\pm k}$
for some $k>0$.
\end{lemma}

\begin{proof}
Consider $f\in G\sm\{1\}$ with $l(f)=l$.
If a normal form for $f$ 
begins with an element in $A\setminus H$
and ends with an element in $B\setminus H$,
then we say that $f$ is of type $AB$.
Similarly, we define the types $BA$, $AA$ and $BB$.
If $l>0$ then $f$ is of type one of the above four types. Let $W=(a^{-1}b)^mf(b^{-1}a)^m$ with $m>l+1$.
If $f$ is of type $AA$, 
then it is trivial that a reduced form of $W$
has the form (\ref{redform}).
We may therefore assume that
$f$ is not of type $AA$.

We first suppose that
$l=0$; thus $f\in H$.
Clearly,
if $b^{\prime}=bfb^{-1}\not\in H$,
then $W\equiv (a^{-1}b)^{m-1}a^{-1}b^{\prime}a(b^{-1}a)^{m-1}$
is a normal form for $W$ and
is of the form (\ref{redform}).
If $b^{\prime}\in H$,
then $b^{\prime}\ne 1$ and thus $a^{-1}b^{\prime}a\in A\setminus H$
because of $(\dagger)$.
Since $m>1$,
we have that $W\equiv (a^{-1}b)^{m-1}a^{\prime}(b^{-1}a)^{m-1}$
is of the form (\ref{redform}), 
where $a^{\prime}=a^{-1}b^{\prime}a$.

Next suppose that $l>0$ and
$f$ is of type $AB$.
In this case, $l\geq 2$.
Let $f=\alpha_1\beta_2\cdots\alpha_{l-1}\beta_{l}$
be a normal form for $f$,
where $\alpha_i\in A\setminus H$ and $\beta_i\in B\setminus H$.
If $\beta_l b^{-1}\in B\setminus H$,
then the assertion is trivial,
and so we may assume that 
$\beta_l b^{-1}\in H$ and also that 
$\alpha_{l-1}^{\prime}=\alpha_{l-1}\beta_l b^{-1}a\in H$.
If $l=2$ and $\alpha_{l-1}^{\prime}=1$,
then $W=(a^{-1}b)^m(b^{-1}a)^{m-1}$,
and hence $W=(a^{-1}b)$.
If $l=2$ and $\alpha_{l-1}^{\prime}\ne 1$,
then $W=(a^{-1}b)^{m}\alpha_{l-1}^{\prime}(b^{-1}a)^{m-1}$.
Since $l(\alpha_{l-1}^{\prime})=0$ and $m-1>2$,
$W'=(a^{-1}b)^{m-1}\alpha_{l-1}^{\prime}(b^{-1}a)^{m-1}$
is of the form (\ref{redform}) and so is $W=a^{-1}bW^{\prime}$.
In the case of $l>2$,
we set $\beta_{l-2}^{\prime}=\beta_{l-2}\alpha_{l-1}^{\prime}$,
$f'= \alpha_1\beta_2\cdots \alpha_{l-3}\beta_{l-2}^{\prime}$,
and $W^{\prime}=(a^{-1}b)^{m-1}f'(b^{-1}a)^{m-1}$.
Since $\beta_{l-2}^{\prime}\in B\setminus H$,
by easy induction on $l$,
we see that the assertion holds for $W^{\prime}$
and so is for $W=a^{-1}bW^{\prime}$.

Similarly, we can prove that
a reduced form of $W$ has the form (\ref{redform}) provided
that $f$ is of type $BB$.
Moreover,
if $f$ is of type $BA$,
then $f^{-1}$ is of type $AB$.
Therefore,
replacing $W$ by $W^{-1}$,
it follows that the assertion holds
when $f$ is of type $BA$.
\end{proof}

\begin{proof}[Proof of Theorem~\ref{PLAFP_TH}]
For some $n\in\N$, say $M=\{f_1,\ldots, f_n\}\sub G$.
By assumption,
there exists a subgroup $N$ with $M\subseteq N$
and $N\simeq A*_{H}B$ which satisfies $(\dagger)$.
As mentioned at the beginning of this section,
it suffices to show that
$M^{x_i}$, over $i\in[3]$, 
are mutually reduced,
where $x_i$ are as in (\ref{xi1}) or (\ref{xi2}).
Replacing $a$ and $a_{*}$ in (\ref{xi1})
by $a^{-1}$ and $a_{*}^{-1}$ respectively,
we can get the case of (\ref{xi2}),
and so we need only show the case (\ref{xi1});
to that end, let $x_i=(b^{-1}a)^{\omega_i}a_{*}b^{-1}a_{*}^{-1}(b^{-1}a)^{\omega_i}$
and assume that $aa_{*}\not\in H$.

Let 
$g_{ip}=x_i^{-1}f_px_i$, for $p\in[n]$, be the elements in $M^{x_i}$.
Since $\omega_i=l+i$ for $i\in[3]$
and $l$ is the maximum number in the set
$\{ l(f_i)\ |\ i\in[n]\}$,
by Lemma \ref{PAFP_L1},
for each $i\in[3]$ and each $p\in[n]$,
a reduced form of
$W_{ip}=(a^{-1}b)^{\omega_i}f_p(b^{-1}a)^{\omega_i}$
has the form 
either 
$W_{ip}\equiv (b^{-1}a)^{\pm k}$ for some $k>0$
or $W_{ip}\equiv(a^{-1}b)V_{ip}(b^{-1}a)$ 
for some non-empty word $V_{ip}$.
In either case,
since $aa_{*}\in A\setminus H$,
a normal form of
$a_{*}^{-1}W_{ip}a_{*}$ is of type $AA$.
We have then that
\begin{equation}\label{gip}
g_{ip}\equiv X_{i}^{-1}A_{ip}X_{i},
\end{equation}
where $X_{i}=b^{-1}a_{*}^{-1}(b^{-1}a)^{\omega_i}$
and $A_{ip}=a_{*}^{-1}W_{ip}a_{*}$.
If $i\ne j$, say $i>j$,
then a normal form of $X_iX_j^{-1}$ is
$b^{-1}a_{*}^{-1}(b^{-1}a)^{\omega_i-\omega_j-1}b^{-1}a'b$
which is of type $BB$,
where $a'=aa^{*}$.
Therefore we have
\begin{equation}\label{gipgjp}
g_{ip}g_{jq}\equiv
X_{i}^{-1}A_{ip}B_{ij}A_{jq}X_{j},
\end{equation}
where $B_{ij}=b^{-1}a_{*}^{-1}(b^{-1}a)^{\omega_i-\omega_j-1}b^{-1}a'b$.

Now,
let $g=g_1\cdots g_k$ be the product of any finite number of
elements $g_i$'s in $\bigcup_{j=1}^{3}M^{x_j}$.
Since a reduced form of $g_i$ has the form (\ref{gip}), 
if both of $g_i$ and $g_{i+1}$ are not in the same $M^{x_j}$ for any $i$,
then by noting that a reduced form of $g_{i}g_{i+1}$
has the form (\ref{gipgjp}),
it can be easily seen by induction on $k$
that $g\equiv X_1^{-1}UX_k$ holds
for some non-empty word $U$ in $G$.
Hence, in particular, $g\ne 1$.
We have thus shown that
$M^{x_i}$'s
are mutually reduced.
\end{proof}

The next corollary improves
\cite[Theorem 3.1 and Remark 3.3]{Bal}.

\begin{corollary}\label{PAFP_C1}
Let $R$ be a domain and $G=A*_{H}B$.
If $G$ satisfies $(\dagger)$ and $|R|\leq |G|$, then the group ring $RG$ is primitive.
In particular,
$KG$ is primitive for any field $K$.
\end{corollary}

\begin{proof}
We need only to show that $G$ has a free subgroup
whose cardinality is the same as that of $G$.
Let $I$ be a set with $|I|=|G|$,
and let $a\in A\setminus H$ such that $a^{-1}Ha\cap H=1$
and $b\in B\setminus H$.
If $|A\setminus H|=|G|$ (resp. $|B\setminus H|=|G|$),
then for each $i\in I$,
there exists $a_i\in A\setminus H$ 
(resp. $b_i\in B\setminus H$)
such that $a_i\ne a_j$ 
(resp. $b_i\ne b_j$) for $i\ne j$.
We have then that
the subgroup of $G$ generated by $a_ib(ab)^2a_ib$ 
(resp. $(ab_i)^3$), over $i\in I$,
is freely generated by them.
On the other hand,
if $|H|=|G|$,
then for each $i\in I$, there is $h_i\in H$ 
with $h_i\ne h_j$ for $i\ne j$.
Let $M_1:=\{\ x_i^{\pm 1}, x_i^{-1}x_j|\ i,j\in I, i\ne j\}$
and $M_2:=\{\ y_i^{\pm 1}, y_i^{-1}y_j|\ i,j\in I, i\ne j\}$
where $x_i:=a^{-1}h_ia$ and $y_i:=b^{-1}a^{-1}h_iab$.
Since, clearly, $M_1$ and $M_2$ are mutually reduced,
it follows from Remark \ref{FSG} that the subgroup of $G$ generated by $z_i=x_iy_i^{-1}$, over $i\in I$,
is freely generated by them.
\end{proof}

We call the free product $A*B$
of two nonidentity groups $A$ and $B$
a strict free product if $A*B\not\simeq\Zh_2*\Zh_2$.
We call $G$ a locally strict free product
if, for each $m\in\N$ and $g_1,\ldots, g_m\in G$,
there exists a subgroup $H$ of $G$ 
which is isomorphic to a strict free product
such that $\{g_1,\ldots, g_m\}\subseteq H$.
Clearly,
if $A*B$ is a strict free product,
then it satisfies $(\dagger)$,
and therefore, 
the following corollary,
which generalizes the result of \cite{For},
follows directly from Theorem \ref{PLAFP_TH}:

\begin{corollary}\label{LFP}
Let $R$ be a domain and $G$ a locally strict free product.
Suppose that $G$ has a free subgroup
whose cardinality is the same as that of $G$.
If $|R|\leq |G|$, then the group ring $RG$ is primitive. In particular,
$KG$ is primitive for any field $K$.
\end{corollary}


\section{Primitivity of group rings of one relator groups with torsion}
\label{OneRGr}

One relator groups,
whose historical origins come from a study
of the fundamental group of a surface,
are perhaps one of the most interesting
and well-studied class of infinite groups.
In particular,
residual finiteness of one relator groups is one of the main topics 
in combinatorial group theory since the 1960s,
where a group is residually finite
provided 
each non-identity element of it can be mapped to a
non-identity element in some homomorphism onto a finite group.
Generally,
one relator groups need not to be residually finite
(see \cite{Bau62}, \cite{B-M-T07}, and \cite{Mesk}).
On the other hand,
it has been conjectured by Baumslag \cite{Bau67}
that every one relator group with torsion
is residually finite,
and it has been believed that the conjecture holds
(see \cite{Wise01} and \cite{Sap-Spa}).
More precisely,
Baumslag conjectured that
one relator groups with torsion
are virtually locally free by cyclic
(see \cite{B-M-T07} and \cite{B-T08}).
If $G$ is a virtually (non-abelian) locally free by cyclic group,
then $KG$ is primitive for any field $K$ by \cite{Ni11}.
We cannot completely settle this conjecture,
but, 
by making use of Theorem \ref{PLAFP_TH} (or Corollary \ref{LFP}),
we can prove that $KG$ is primitive when $G$ is 
a one relator group with torsion:

\begin{theorem}\label{ORG_TH}
If $G$ is a non-cyclic one relator group with torsion,
then $KG$ is primitive for any field $K$.
\end{theorem}

In order to prove the above theorem,
we prepare necessary notation 
and group theoretic results on one relator groups with torsion. Throughout this section,
$F=\langle X\rangle$ denotes the free group
with a base $X$.
Let $G=\langle X\ | R\rangle$ denote the one relator group
with the set of generators $X$
with a relation $R$,
where $R$ is a cyclically reduced word in $F$.
For a word $W$ in $F$,
if $R=W^n$, $n>1$ and $W$ is not a proper power in $F$,
then $G$ is called a one relator group with torsion. 

Let $W$ be a word in $F$.
We denote the normal closure of $W$ in $F$ by $\Nc_{F}(W)$.
For a cyclically reduced word $W$,
$\Wsy_{F}(W)$ denotes the set
of all cyclically reduced conjugates of both $W$ and $W^{-1}$.
If $W_i,\ldots, W_t$ are reduced words in $F$
and $W=W_i\cdots W_t$ is also reduced,
that is, there is no cancellation in forming the product $W_i\cdots W_t$,
then we write $W\equiv W_i\cdots W_t$.
For $Y\subseteq X$,
$\langle Y\rangle_G$ is the subgroup of $G$
generated by the homomorphic image in $G$ of $Y$.

\begin{lemma}\label{L_ST}
Let $n\in\N$ and $G=\langle X\ | R\rangle$,
where $W$ is a cyclically reduced word in $F$
and $R=W^n$.
\begin{enumerate}[(1)]
\item 
\mbox{{\rm (\cite[Theorem]{Sh}, cf. \cite{HP})}}\
If $V\in \Nc_{F}(R)\sm\{1\}$,
then $V$ contains a subword $S^{n-1}S_0$,
where $S\equiv S_0S_1\in \Wsy_{F}(W)$ and every
generator appearing in $W$ appears in $S_0$.
\end{enumerate}
\begin{enumerate}[(2)]
\item 
\mbox{{\rm (\cite[Theorem]{New73})}}\
The centralizer of every non-trivial element in $G$ is cyclic. 
\end{enumerate}
\end{lemma}

The next two results in the lemma below are probably well-known to experts,
but we include their proofs for completeness.

\begin{lemma}\label{L_S}
For $n>1$,
let $G=\langle X \ |\ R \rangle$ with $|X|>1$,
where $R=W^n$ 
and $W$ is a cyclically reduced word in $F$.
\begin{enumerate}[(1)]
\item 
If $S, T\subseteq X$,
then $\langle S\rangle_{G}\cap \langle T\rangle_{G}
=\langle S\cap T\rangle_{G}$.
\end{enumerate}
\begin{enumerate}[(2)]
\item 
$\Delta(G)=1$.
\end{enumerate}
\end{lemma}
%

\begin{proof}
(1):\
If $S\subseteq T$ or $T\subseteq S$,
then the assertion is clear,
and so we may assume $S\not\subseteq T$ and $T\not\subseteq S$.
It is obvious that $\langle S\rangle_{G}\cap \langle T\rangle_{G}
\supseteq\langle S\cap T\rangle_{G}$.
Suppose, to the contrary, that $\langle S\rangle_{G}\cap \langle T\rangle_{G}
\supsetneq\langle S\cap T\rangle_{G}$.
Then there exist reduced words $u=u(s,a,\ldots,b)$
in $\langle S\rangle\setminus\langle S\cap T\rangle$
and $v=v(t,c,\ldots,d)$
in $\langle T\rangle\setminus\langle S\cap T\rangle$
such that $uv\in\Nc_{F}(R)$,
where $a,\ldots,b\in S$, $c,\ldots,d\in T$, $s\in S\setminus (S\cap T)$,
and $t\in T\setminus (S\cap T)$.
Let $w$ be the reduced word for $uv$,
say $w\equiv u_1v_1$, where $u\equiv u_1u_2$ and $v\equiv u_2^{-1}v_1$.
Then $w\equiv u_1v_1\in \Nc_{F}(R)$.
However, $u_1$ involves $s$ but not $t$,
and $v_1$ involves $t$ but not $s$,
which contradicts the assertion of Lemma \ref{L_ST} (1).

(2):\
Suppose , to the contrary, $\Delta(G)\ne 1$;
then, there exists $1\ne g\in G$ such that $[G:C_{G}(g)]<\infty$.
By Lemma \ref{L_ST} (2),
$C_{G}(g)$ is cyclic and in fact infinite cyclic
because $|G|$ is not finite.
Thus $G$ is virtually cyclic
and so, as is well-known,
there exists a normal subgroup $N$ of finite order
such that $G/N$ is isomorphic to either 
the infinite cyclic group $\Zh$
or the infinite dihedral group $\Zh_2*\Zh_2$ (See \cite[p137]{JL}).

Since a one relator group with torsion is not isomorphic to 
$\Zh$ or $\Zh_2*\Zh_2$,
we may assume $N\ne 1$.
In both cases of $G/N\simeq \Zh$ and $G/N\simeq\Zh_2*\Zh_2$,
there exists $x\in G\setminus N$
such that $\langle x\rangle_G$ is 
a infinite cyclic subgroup of $G$.
Since $|N|$ is finite,
then it is easily seen that
there exists $m>0$ such that $x^{-m}fx^{m}=f$ for all $f\in N$,
which implies $N\subseteq C_{G}(x^m)$; this is a contradiction, 
since a infinite cyclic group does not contain
non-trivial finite subgroups.
\end{proof}

Let $X=\{x_1,x_2,\ldots , x_m\}$ with $m>1$
and $F=\langle X\rangle$.
To avoid unnecessary subscripts, 
we denote generators, $x_1,x_2,\ldots, x_m$,
by $t, a,\ldots, b$.
We consider the one relator group $G=\langle X\ | R\rangle$,
where $R=W^n$, $n>1$ and $W=W(t,a,\ldots, b)$ is a cyclically reduced word
which is not a proper power.
We assume that all generators appear in $W$.
We shall see that 
there exists a normal subgroup $L$ of $G$ such that
$G/L$ is cyclic and $L$ satisfies 
the assumption in
Corollary \ref{LFP}. 
That is,
$G$ has the following type of subgroup $G_{\infty}$
and $L$ is a subgroup of it, which shall be shown in
Proof of Theorem \ref{ORG_TH} below:
\begin{equation}\label{D_-1}
G_{\infty}=\langle X_{\infty}\ |\ R_i,\ i\in \Zh\rangle
\mbox{ with } R_i=W_i^{n} (n>1),
\end{equation}
where $X_{\infty}=\{a_j,\ldots, b_j \ |\ j\in \Zh\}$
and for each $i\in \Zh$,
$W_i$ is a cyclically reduced word in the free group,
which is as follows:
$F_{\infty}=\langle X_{\infty}\rangle$.
Let $\alpha_{*}$,$\ldots$, $\beta_{*}$
be  respectively the minimum subscripts 
on $a$, $\ldots$, $b$ occurring in $W_0$,
and let $\alpha^{*}, \ldots, \beta^{*}$ be the maximum subscript 
on $a$,$\ldots$, $b$ occurring in $W_0$, respectively.
Then $W_i$ is a word expressed by
$$W_i=W_i(a_{\alpha_{*}+i},\ldots,a_{\alpha^{*}+i}, 
\ldots, b_{\beta_{*}+i},\ldots, b_{\beta^{*}+i}).$$

Let $\mu$ be the maximum number 
in $\{\alpha^{*}-\alpha_{*}, \ldots, \beta^{*}-\beta_{*}\}$.
For $t\in \Zh$, 
we set subgroups $Q_t$ and $P_t$ of $G_{\infty}$ as follows:
\begin{equation}\label{D0}
\left\{\begin{array}{llll}
\begin{array}{llll}
\t{If }\mu\neq 0,\,\hspace{0.1in}Q_t:=\langle a_{t+i},  \ldots, b_{t+j} 
\ |\ \alpha_{*}\leq i\leq\alpha^{*},\ \ldots,\ \beta_{*}\leq j\leq\beta^{*}
\rangle_{G_{\infty}}\t{ and}\\
\hspace{0.67in}\hspace{0.1in}P_t:=\langle a_{t+i},  \ldots, b_{t+j}
\ |\ \alpha_{*}\leq i\leq\alpha^{*}-1,\ \ldots,\ \beta_{*}\leq j\leq\beta^{*}-1
\rangle_{G_{\infty}}.\\
\end{array}\\
\begin{array}{llll}
\t{If }\mu= 0,\,\hspace{0.1in}Q_t:=\langle a_{t+\alpha_{*}},\ldots,\ b_{t+\beta_{*}} 
\rangle_{G_{\infty}}\t{ and}\\
\hspace{0.66in}\hspace{0.1in}P_t:=1.\\
\end{array}\\
\end{array}\right.
\end{equation}
Then $P_t$ is a subgroup of $Q_t$ 
and $Q_t$ has the following presentation:
\begin{equation}\label{Qt}
Q_t\simeq \langle a_{t+\alpha_{*}}, 
\ldots, a_{t+\alpha^{*}},\ \ldots,\
b_{t+\beta_{*}},
\ldots, b_{t+\beta^{*}}\ |\ R_t\rangle.
\end{equation}
In what follows, let $\nu:=\beta^{*}-\beta_{*}$; replacing the order of $a_{i},\ldots, b_{i}$ in $X_{\infty}$
if necessary, assume $\mu=\alpha^{*}-\alpha_{*}\geq \cdots
\geq \beta^{*}-\beta_{*}=\nu$.
With Magnus' method for Freiheitssatz,
we may identify $G_{\infty}$
as the union of the following chain
(see \cite{MKS} or \cite{Lyn}):
\begin{equation}\label{D1}
\begin{array}{llll}
G_{\infty}=\bigcup_{i=0}^{\infty}G_i, \mbox{ where}\\
G_0=Q_0, \quad G_{2i}=Q_{-i}*_{P_{-i+1}}G_{2i-1},
\mbox{ and }\ G_{2i+1}=G_{2i}*_{P_{i+1}}Q_{i+1}.\\
\end{array}
\end{equation}

\begin{lemma}\label{L_D}
If $H$ is a subgroup of $G_{\infty}$
generated by a finite subset $Y$ of $X_{\infty}$; 
namely $H=\langle Y \rangle_{G_{\infty}}$,
then there is a positive integer $t$
so that $H\subseteq G_{2(t-1)}$ and
$H\cap P_t=1$.
\end{lemma}

\begin{proof}
Since $G_{\infty}=\bigcup_{i=0}^{\infty}G_i$
and $G_0\subsetneq G_1\subsetneq\cdots\subsetneq G_{2i}\subsetneq G_{2i+1}\subsetneq\cdots,$
there exists some $s\geq 0$ such that $G_{2s}\supseteq H$.
Let $t$ be a positive integer satisfying
\begin{equation}\label{D4}
s+\alpha^{*}<t+\alpha_{*}\,,\ldots,\
s+\beta^{*}<t+\beta_{*}.
\end{equation}
Since $s\leq t-1$ and $H\subseteq G_{2(t-1)}$, to finish the proof,
it suffices to show that $H\cap P_t=1$.
If $\mu=0$ in (\ref{D0}),
then the assertion is trivial, so assume that $\mu\ne 0$.

Suppose, to the contrary,
there exists $t\in\N$ which
satisfies (\ref{D4}) and $H\cap P_t\ne 1$.
For brevity,
we write $\hat{a}_i$ and $\hat{b}_i$ 
instead of $a_{t+\alpha_{*}+i}$ and $b_{t+\beta_{*}+i}$, respectively;
namely,
$P_t=\langle \hat{a}_{0},\ldots, \hat{a}_{\mu-1},\ \ldots,\
\hat{b}_{0}, 
\ldots, \hat{b}_{\nu-1} \rangle_{G_{\infty}}$.
For $j\in\{0,1,\ld,\mu\}$,
define $P_{t}^{(j)}$ so that
$$P_{t}=P_{t}^{(\mu)}\supsetneq P_{t}^{(1)}\supsetneq \cdots \supsetneq P_{t}^{(0)}=1$$
as follows:
$$\begin{array}{rll}
P_{t}=&P_{t}^{(\mu)}
&=\langle \hat{a}_{0},\ldots, \hat{a}_{\mu-1},\ \ldots,\
\hat{b}_{0}, 
\ldots, \hat{b}_{\nu-1} \rangle_{G_{\infty}}\\
&P_{t}^{(\mu-1)}
&=\langle \hat{a}_{0},\ldots, \hat{a}_{\mu-2},\ \ldots,\
\hat{b}_{0}, 
\ldots, \hat{b}_{\nu-2} \rangle_{G_{\infty}},\\
&\vdots & \vdots\\
&P_{t}^{(\mu-\nu+1)}
&=\langle \hat{a}_{0},\ldots,a_{\mu-\nu},\
\ldots,\
\hat{b}_{0}\rangle_{G_{\infty}}, \\
&P_{t}^{(\mu-\nu)}
&=\langle \hat{a}_{0},\ldots,\hat{a}_{\mu-\nu-1},
\ldots \rangle_{G_{\infty}}, \\
&\vdots & \vdots\\
&P_{t}^{(1)}
&=\langle \hat{a}_{0}\rangle_{G_{\infty}},\\
&P_{t}^{(0)}
&=1.\\
\end{array}$$
That is,
generators in $\{\hat{a}_{0},\ldots, \hat{a}_{\mu-1}\},\ldots$ 
and in $\{\hat{b}_{0}, \ldots, \hat{b}_{\nu-1}\},$
are respectively decremented one by one
from $P_{t}^{(\mu)}$ to $P_{t}^{(0)}$.
By our assumption,
$H\cap P_t\ne 1$, i.e.,
there is $u\in H\cap P_t$ with $u\ne 1$.
Thus, there is $l\in\{0,1,\ldots,\mu-1\}$ with
$u\in P_{t}^{(\mu-l)}$ and $u\not\in P_{t}^{(\mu-l-1)}$.
We shall see that this is impossible.
In fact, we shall show that $u\in H\cap P_{t}^{(\mu-l)}$
implies $u\in P_{t}^{(\mu-l-1)}$,
which completes the proof of the lemma.

Let $u\in H\cap P_{t}^{(\mu-l)}$.
By (\ref{D4}),
$s\leq t-\mu-1\leq t-l-2$,
which implies
\begin{equation}\label{D5}
H\subseteq G_{2(t-l-2)}
\end{equation}
because $H\subseteq G_{2s}\subseteq G_{2(t-l-2)}$.
By the construction of $P_{t}^{(\mu-l)}$, 
the set $T$ of generators of $P_{t}^{(\mu-l)}$ is
$$T=\{\hat{a}_{0}, \ldots, \hat{a}_{\mu-l-1},\ 
\ldots,\
\hat{b}_{0}, \ldots, \hat{b}_{\nu-l-1}\},$$
where for instance, 
$\{\hat{b}_{0}, \ldots, \hat{b}_{\nu-l-1}\}$
is $\emptyset$ if $\nu-l-1<0$.
By (\ref{D0}), the generators of $Q_{t-l-1}$ are
$$\hat{a}_{-l-1}, \ldots, \hat{a}_{0}, \ldots, \hat{a}_{\mu-l-1},\
\ldots,\
\hat{b}_{-l-1}, \ldots, \hat{b}_{0},\ldots,
\hat{b}_{\nu-l-1},$$
and therefore we see that $P_{t}^{(\mu-l)}\subseteq Q_{t-l-1}$.
Combining this with (\ref{D5}), it follows that
$u\in G_{2(t-l-2)}\cap Q_{t-l-1}$. 
Since $G_{2(t-l-2)}\cap Q_{t-l-1}=P_{t-l-1}$,
we have $u\in P_{t-l-1}$,
and thus $u\in P_{t-l-1}\cap P_{t}^{(\mu-l)}$.

On the other hand, the set $S$ of generators of $P_{t-l-1}$ in $Q_{t-l-1}$
is 
$$\begin{array}{lll}
&S=\{\hat{a}_{-l-1}, \ldots, \hat{a}_{\mu-l-2},\
\ldots,\
\hat{b}_{-l-1}, \ldots, \hat{b}_{\nu-l-2}\}.\\
\end{array}$$
Thus it is easy to see that
$\langle S\cap T \rangle_{Q_{t-l-1}}=P_{t}^{(\mu-l-1)}$.
We may regard $Q_{t-l-1}$ as a one-relator group with torsion,
and therefore it follows form Lemma \ref{L_S} (1) that
$$u\in P_{t-l-1}\cap P_{t}^{(\mu-l)}
=\langle S\rangle_{Q_{t-l-1}}\cap \langle T\rangle_{Q_{t-l-1}}
=\langle S\cap T \rangle_{Q_{t-l-1}}=P_{t}^{(\mu-l-1)};$$
thus $u\in P_{t}^{(\mu-l-1)}$, as desired.
\end{proof}

\begin{lemma}\label{OG_LFP}
If $G_{\infty}$ and $W_i$ are as in (\ref{D_-1}),
then for each $m\in\N$ and $g_1,\ldots,g_m\in G_{\infty}$,
there is $t\in\N$ with
$\langle g_1,\ldots,g_m, W_t\rangle_{G_{\infty}}= \langle g_1,\ldots,g_m\rangle_{G_{\infty}}*\langle W_t\rangle_{G_{\infty}}$.
\end{lemma}

\begin{proof}
Let $Y$ be the subset of $X_{\infty}$
consisting of generators 
appearing in $g_i$ for $i\in [m]$.
By Lemma \ref{L_D}, if $H:=\langle Y\rangle_{G_{\infty}}$,
there is $t\in\N$ with $H\subseteq G_{2(t-1)}$ and $H\cap P_t=1$. By (\ref{D1}),
$G_{2t-1}=G_{2(t-1)}*_{P_t}Q_{t}$,
where 
$Q_t$ is as described in (\ref{Qt})
and $P_t$ is as described in (\ref{D0}).
Since $W_t^n=R_t$ is the relator of $Q_t$,
we have $\langle W_t\rangle_{G_{\infty}}\subseteq Q_t$.
As is well known,
$W_t^m\ne 1$ in $Q_t$ for $m\in[n-1]$.
Moreover, $P_{t}\cap\langle W_t\rangle_{Q_{t}}=1$.
In fact, if not, there would be $m\in\N_0$
with $W_t^m\in P_t$ in $Q_t$.
Since $P_t$ is a free subgroup of $Q_t$
by Freiheitssatz,
we have that $1\ne (W_t^m)^n=(W_t^n)^m$ in $Q_t$.
However,
this contradicts that $W_t^n$ is the relator of $Q_t$.
We have thus shown that
$P_{t}\cap\langle W_t\rangle_{Q_{t}}=1$.
Combining this with $H\cap P_t=1$,
we see that
$\langle Y, W_t\rangle_{G_{2t-1}}
=\langle Y\rangle_{G_{2t-1}}*\langle W_t\rangle_{G_{2t-1}}$
$=H*\langle W_t\rangle_{G_{\infty}}$.
Since $\langle g_1,\ldots,g_m\rangle_{G_{\infty}}\subseteq H$,
we have that 
$\langle g_1,\ldots,g_m, W_t\rangle_{G_{\infty}}
=\langle g_1,\ldots,g_m\rangle_{G_{\infty}}*\langle W_t\rangle_{G_{\infty}}$.
\end{proof}

We are now in a position to prove Theorem \ref{ORG_TH}.

\begin{proof}[Proof of Theorem~\ref{ORG_TH}]
Let $G=\langle X\ | R\rangle$ be the one relator group with torsion,
where $|X|>1$, $R=W^n$, $n>1$ and $W$ is a cyclically reduced word
which is not a proper power.
If there exists a generator $x\in X$ which does not appear in $W$,
then $G$ is isomorphic to the free product
$\langle x\rangle*\langle X\setminus \{x\}\ |\ R\rangle$,
and so $KG$ is primitive for any field $K$
by Corollary \ref{PAFP_C1} or by the result of
Formanek \cite{For}.
Hence we may assume that $X$ is a finite set
and all generators in $X$ appear in $W$.
Let $X=\{t,a,b,\ldots,c\}$
and $W=W(t,a,b,\ldots,c)$.

In this case,
the cardinality of $G$ is countable,
and it is well-known that $G$ has a non-cyclic free subgroup.
Moreover, by Lemma \ref{L_S} (2), we see that $\Delta(G)=1$,
and therefore, combining Corollary \ref{LFP} 
with
Lemma \ref{PBR},
it suffices to show that there exists a normal
subgroup $L$ of $G$ such that $G/L$ is cyclic and $L$ satisfies 
the following condition:
\vskip3pt

\begin{tabularx}{13cm}{lX}
$(\star)$&
For any
$g_1,\ldots,g_m\in L$, 
there exists a free product
$A*B$ in the set of subgroups of $L$ such that $B\ne 1, a^2\ne 1$
for some $a\in A$, and $g_1,\ldots,g_m\in A*B$.\\
\end{tabularx}
\vskip3pt

There are now two cases to consider:
whether or not the exponent sum $\sigma_{x}(W)$ of 
$W$ on some generator $x$ is zero. If the exponent sum $\sigma_{x}(W)$ of 
$W$ on some generator $x$ is zero, say $\sigma_{t}(W)=0$,
then we set $N=\Nc_{F}(R)$ and $M=\Nc_{F}(a,b,\ldots,c)$.
In this case,
$G\simeq F/N$, where $F=\langle X\rangle$.
Since $\sigma_{t}(W)=0$,
we have that $N\subsetneq M$.
By making use of
a Reidemeister-Schreier rewriting process,
we get a presentation of $G_{\infty}=M/N$ as follows:
$$
G_{\infty}=\langle X_{\infty}\ |\ R_i,\ i\in \Zh\rangle
\mbox{ with } R_i=W_i^{n} (n>1),
$$
where $X_{\infty}=\{a_i,b_i,\ldots, c_i \ |\ i\in \Zh\}$,
$a_i=t^iat^{-i},b_i=t^ibt^{-i},\ldots,c_i=t^ict^{-i}$,
and $W_i=t^iWt^{-i}$ $(i\in \Zh)$.
Let $L=G_{\infty}$.
Then
$L$ is a normal subgroup of $G$ and $G/L$ is cyclic.
We can see that $L$ satisfies
the $(\star)$.
In fact, in (\ref{D0}),
if $\mu=0$ then the subgroup $G_1$ of $L(=G_{\infty})$
is the free product $Q_0*Q_1$.
If $\mu\ne 0$ then $G_1=Q_0*_{P_1}Q_1$.
In the former case,
$a_{\alpha_{*}}=a_{\alpha^{*}}$,
in the latter case,
$a_{\alpha_{*}}, a_{\alpha^{*}+1}\not\in P_1$,
and in either case,
$a_{\alpha_{*}}\in Q_0$ and $a_{\alpha^{*}+1}\in Q_1$.
Let $u=a_{\alpha_{*}}a_{\alpha_{*}+1}$.
Then
$\langle u\rangle$
is an infinite cyclic subgroup of $L$.
In particular, $u^2\ne 1$ in $L$.
By Lemma \ref{OG_LFP},
for any finite number of elements $g_1,\ldots,g_m$ and for $u$,
there exists $t>0$ such that
$\langle u, g_1,\ldots, g_m\rangle_{L}
*\langle W_t\rangle_{L}$,
and thus $L$ satisfies the $(\star)$
because $W_t\ne 1$ and $u^2\ne 1$,
as desired.

On the other hand,
if for each $x\in X$, $\sigma_{x}(W)\ne 0$,
say $\sigma_{t}(W)=\beta$ and $\sigma_{a}(W)=\gamma$,
then, replacing $t$ by $t^{\gamma}$ both in $X$ and in $R$,
we define $\widehat{G}$ by $\langle \widehat{X}\ |\ \widehat{R}\rangle$,
where $\widehat{X}=\{t^{\gamma}, a,b,\ldots, c\}$,
$\widehat{R}=(\widehat{W})^n$ 
and $\widehat{W}=\widehat{W}(t^{\gamma}, a,b,\ldots,c)$.
We set $\widehat{F}=\langle \widehat{X}\rangle$,
$N=\Nc_{F}(\widehat{R})$
and $\widehat{N}=\Nc_{\widehat{F}}(\widehat{R})$.
We note that $\widehat{F}$, $N$ and $\widehat{N}$
are subgroups of $F=\langle X\rangle$.
Clearly,
$\widehat{G}$ is isomorphic to $G$.
We may regard $\widehat{G}$
as a subgroup of $\langle X\ |\ \widehat{R}\rangle\simeq F/N$.
Since
$\widehat{F}/\widehat{N}\simeq \widehat{G}\simeq G$,
we identify $\widehat{G}$ with $\widehat{F}/\widehat{N}$
and it suffices to show that $\widehat{G}$ has a normal subgroup $\widehat{L}$
which satisfies the $(\star)$
and $\widehat{G}/\widehat{L}$ is cyclic.

Now,
let $M=\Nc_{F}(at^{\beta},b,\ldots,c)$.
Since $\sigma_{t}(W)=\beta$ and $\sigma_{a}(W)=\gamma$,
we see that
$$\sigma_{t}(\widehat{W})=\beta\gamma=\beta\sigma_{a}(W)
=\beta\sigma_{a}(\widehat{W}),$$
which implies $\widehat{W}\in M$ and thus $N\subsetneq M$.
Similarly to the previous case,
we have a presentation of 
$G_{\infty}=M/N$ as follows:
$$
G_{\infty}=\langle X_{\infty}\ |\ \widehat{R}_i,\ i\in \Zh\rangle
\mbox{ with } \widehat{R}_i=(\widehat{W}_i)^{n} (n>1),
$$
where $X_{\infty}=\{a_i,b_i,\ldots, c_i \ |\ i\in \Zh\}$,
$a_i=t^iat^{\beta-i},b_i=t^ibt^{-i},\ldots,c_i=t^ict^{-i}$,
and $\widehat{W}_i=t^i\widehat{W}t^{-i}$ $(i\in \Zh)$.
In this case,
$G_{\infty}$ is not a subgroup of $\widehat{G}$,
and therefore,
we let $\widehat{L}=(M\cap \widehat{F})/\widehat{N}$.
Then $\widehat{L}$ is a normal subgroup of $\widehat{G}$.
Since 
$\widehat{G}/\widehat{L}\simeq \widehat{F}/(M\cap \widehat{F})
\simeq \widehat{F}M/M\subsetneq F/M\simeq \langle t\rangle,$
$\widehat{G}/\widehat{L}$ is cyclic.
To finish the proof,
it remains to show that $\widehat{L}$ satisfies
the $(\star)$.

Since $\widehat{G}$
is isomorphically embedded into $F/N$, it is clear that $\widehat{F}\cap N=\widehat{N}$, so
$$\begin{array}{lll}
G_{\infty}=M/N &\supset (M\cap \widehat{F})N/N
&\simeq (M\cap \widehat{F})/(M\cap \widehat{F}\cap N)\\
&&= (M\cap \widehat{F})/(\widehat{F}\cap N)\\
&&=(M\cap \widehat{F})/\widehat{N}=\widehat{L}.\\
\end{array}$$
Hence we may assume that $\widehat{L}$ is a subgroup of $G_{\infty}$.
Let $g_1,\ldots,g_m$ $(m>0)$ be in $\widehat{L}$ with $g_i\ne 1$.
In case of $n>2$,
since $\widehat{L}\subsetneq G_{\infty}$,
by Lemma \ref{OG_LFP},
there exists $t>0$ such that
$\langle g_1,\ldots, g_m\rangle_{G_{\infty}}
*\langle \widehat{W}_t\rangle_{G_{\infty}}$.
We have then that 
$1\ne \widehat{W}_t\in \widehat{L}$ and $(\widehat{W}_t)^2\ne 0$
because $n>2$,
and so $\widehat{L}$ satisfies
the $(\star)$.
On the other hand, in case of $n=2$,
let $p>0$ be the maximum number such that either
$t^{p\gamma}$ or $t^{-p\gamma}$
is appeared in $\widehat{W}=\widehat{W}(t^{\gamma}, a,b,\ldots,c)$.
Set $v=t^{(p+1)\gamma}at^{-(p+1)\gamma}a^{-1}$
so that $v\in \widehat{F}$.
Moreover,
since $\sigma_t(v)=0$ and $\sigma_a(v)=0$,
we have $v\in M$.
That is, $v\in M\cap \widehat{F}$ and thus
the homomorphic image $\overline{v}$ of $v$ is contained in $\widehat{L}$.
Suppose that $\overline{v}^2=1$;
namely, $v^2\in \widehat{N}$.
In view of Lemma \ref{L_S} (1),
a reduced word $v^2$ 
contains a subword $S_0S_1S_0$
such that $S_0S_1$ is 
a cyclic shift of $\widehat{W}$
and $S_0$ contains all generators appeared in $\widehat{W}$.
Since only two letters $t$ and $a$ are appeared in $v^2$,
we have that  $\widehat{W}=\widehat{W}(t^{\gamma}, a)$.
Moreover, $S_0S_1S_0$
involves a subword of type $a^{\varepsilon_1}t^{q}a^{\varepsilon_2}$
with $|q|\leq |p\gamma|$, where $\varepsilon_i=\pm 1$.
However, since $|(p+1)\gamma|>|q|$,
there exists no such subword in $v^2$,
which implies a contradiction.
We have thus shown that $\overline{v}^2\ne 1$.
By Lemma \ref{OG_LFP},
for $g_1,\ldots,g_m$ and $\overline{v}$,
there exists $t>0$ such that
$\langle \overline{v}, g_1,\ldots, g_m\rangle_{G_{\infty}}
*\langle \widehat{W}_t\rangle_{G_{\infty}}$.
Since $1\ne \widehat{W}_t\in \widehat{L}$ and $\overline{v}^2\ne 1$,
we have thus proved that $\widehat{L}$ satisfies condition $(\star)$.
\end{proof}


\end{document}